\documentclass{amsart}
\usepackage{graphicx}
\usepackage{amssymb}
\usepackage{amsfonts}
\usepackage{hyperref}
\usepackage{mathrsfs}
  [1995/08/01 v0.1 Double stroke roman fonts]

\def\ds@whichfont{dsrom}
\relax

\DeclareMathAlphabet{\mathds}{U}{\ds@whichfont}{m}{n}

\swapnumbers
\newtheorem{theorem}{Theorem}[section]
\newtheorem{lemma}[theorem]{Lemma}
\newtheorem{corollary}[theorem]{Corollary}

\theoremstyle{definition}

\newtheorem{assumption}[theorem]{Assumption}
\newtheorem{remark}[theorem]{Remark}

\numberwithin{equation}{section}
\theoremstyle{plain}

\numberwithin{equation}{section} %% Comment out for sequentially-numbered
\numberwithin{figure}{section} %% Comment out for sequentially-numbered
\theoremstyle{plain}
\theoremstyle{plain}
\theoremstyle{remark}
\newtheorem*{acknowledgement*}{Acknowledgement}
\newcommand{\cA}{{\mathcal A}}

\newcommand{\cC}{{\mathcal C}}
\newcommand{\cD}{{\mathcal D}}

\newcommand{\cF}{{\mathcal F}}
\newcommand{\cG}{{\mathcal G}}
\newcommand{\cH}{{\mathcal H}}
\newcommand{\cI}{{\mathcal I}}
\newcommand{\cJ}{{\mathcal J}}

\newcommand{\cL}{{\mathcal L}}

\newcommand{\cN}{{\mathcal N}}

\newcommand{\cT}{{\mathcal T}}

\newcommand{\cX}{{\mathcal X}}

\newcommand{\te}{{\theta}}

\newcommand{\Om}{{\Omega}}
\newcommand{\om}{{\omega}}
\newcommand{\ve}{{\varepsilon}}
\newcommand{\del}{{\delta}}
\newcommand{\Del}{{\Delta}}
\newcommand{\gam}{{\gamma}}
\newcommand{\Gam}{{\Gamma}}

\newcommand{\sig}{{\sigma}}
\newcommand{\al}{{\alpha}}
\newcommand{\be}{{\beta}}
\newcommand{\ka}{{\kappa}}
\newcommand{\la}{{\lambda}}

\newcommand{\bbE}{{\mathbb E}}

\newcommand{\bbN}{{\mathbb N}}

\newcommand{\bbR}{{\mathbb R}}

\newcommand{\bbZ}{{\mathbb Z}}

\newcommand{\brF}{{\bar F}}

\begin{document}
\title[]{Nonconventional moderate deviations theorems and exponential concentration inequalities}%The method of cumulants for nonconventional sums and applications to...?
 \vskip 0.1cm
 \author{Yeor Hafouta \\
\vskip 0.1cm
 Institute  of Mathematics\\
Hebrew University\\
Jerusalem, Israel}%
\address{
Institute of Mathematics, The Hebrew University, Jerusalem 91904, Israel}
\email{yeor.hafouta@mail.huji.ac.il}%

\thanks{ }
\dedicatory{  }
 \date{\today}

\begin{abstract}\noindent
We obtain moderate deviations theorems and exponential (Bernstein type) concentration inequalities
for ``nonconventional" sums of the form 
$S_N=\sum_{n=1}^N (F(\xi_{q_1(n)},\xi_{q_2(n)},...,\xi_{q_\ell(n)})-\bar F)$, where most of the time we consider $q_i(n)=in$, but our results also hold true for more general  $q_i(n)$'s such as polynomials.
Here $\xi_n,\, n\geq 0$ is a sufficiently fast mixing vector
 process with some stationarity conditions, $F$ is a function satisfying certain regularity
 conditions and
$\bar F$ is a certain centralizing constant. When $\xi_n,\,n\geq 0$ are independent and
identically distributed a large deviations theorem was obtained in \cite{KV1} and one of the purposes
of this paper is to obtain related results in the (weakly) dependent case. 
Several normal approximation type results will also be derived. In particular,
two more proofs of the nonconventional central limit theorem are given and a Rosenthal type inequality is obtained.
Our results hold true, for instance, when $\xi_n=(T^nf_i)_{i=1}^\wp$ where $T$ is a topologically 
mixing subshift of finite type, a Gibbs-Markov map, a hyperbolic diffeomorphism, a Young tower or an expanding transformation taken with a Gibbs invariant measure, as well as in the case when $\xi_n,\, n\geq 0$ forms
a stationary and (stretched) exponentially fast $\phi$-mixing sequence, which, for instance, holds true
when $\xi_n=(f_i(\Upsilon_n))_{i=1}^\wp$ where $\Upsilon_n$ is a Markov
chain satisfying the Doeblin condition considered as a stationary
process with respect to its invariant measure.

\end{abstract}
%\footnotetext[1]{}
\maketitle
\markboth{Y. Hafouta}{Nonconventional moderate deviations theorems and exponential concentration inequalities}
\renewcommand{\theequation}{\arabic{section}.\arabic{equation}}
\pagenumbering{arabic}

\section{Introduction}\label{sec1}\setcounter{equation}{0}
Partially motivated by the research on nonconventional ergodic theorems (the term ``nonconventional" 
comes from \cite{Fu}), probabilistic limit theorems for sums of the form 
$S_N=\sum_{n=1}^N\big(F(\xi_{q_1(n)},\xi_{q_2(n)},...,\xi_{q_\ell(n)})-\bar F\big)$ have 
become a well  studied topic. Here $\xi_n,\, n\geq 0$ is a sufficiently fast mixing vector process with some 
stationarity properties, $F$ is a function satisfying some regularity conditions and $\bar F$ is a certain 
centralizing constant.
During the past decade many of the classical results such as the (functional) central
limit theorem, Berry-Esseen type theorem, the local central limit theorem, Poissonian limit theorems and large deviations 
theorems were obtained for such sums (see \cite{Ki2},\cite{KV},\cite{KV1},\cite{book} and references therein).
One of the most interesting choices of $q_i$'s is the situation when $q_i(n)=in$ for any $i=1,2,...,\ell$. 
This was the original motivation for the study of nonconventional sums and yields appropriate limit
theorems for number of multiple recurrencies to a given set by $\xi_k$'s at times forming arithmetic progressions of the type $n,2n,...,\ell n$.

The large deviation priciple proved in \cite{KV1} holds true in the case when 
$S_N=\sum_{n=1}^N\big(F(\xi_n,\xi_{2n},...,\xi_{\ell n})-\bar F)$ only for independent and identically distributed 
$\xi_n$'s, while when the $q_i(n)$'s satisfy certain (faster than linear) growth conditions the results from there
hold true also for certain Markov chains and dynamical systems.
The main goal of this paper is to obtain related results when the $\xi_n$'s are weakly
dependent and not necessarily generated by a Markov chain or a dynamical system. 
We will first obtain moderate deviation type theorems for such sums, namely, study the asymptotic behaviour as $N\to\infty$ of probabilities of the form 
\[
P\big(\frac1{N^\zeta}S_N\in\Gam\big)
\]
for arbitrary Borel measurable sets $\Gam\subset\bbR$.
Here $\frac12<\zeta<1$ depends on the amount of regularity of $F$ and on 
the growth of $\bbE|\xi_1|^k$ as $k\to\infty$.
Formally (see \cite{DemZet}), any choice of $\zeta$
is considered as large deviations type result, but
under our conditions $\frac 1NS_N$  will satisfy the
law of large numbers (see \cite{Kif-LLN}) and so we will use the standard informal convention of referring
to the case when $\zeta=1$ as the large deviations case, while the case when $0<\zeta<1$ will be referred to  
as the moderate deviations case, where in our situation it is natural to require that $\frac12<\zeta$ since $N^{-\frac12}S_N$ satisfies the central limit theorem (see \cite{KV} and \cite{HK2}). 
Exponential concentration inequalities (i.e. estimates of $P(S_N\geq x),\,x>0$)
and Gaussian type estimates of the moments of $S_N$ will also be derived. 
All of the above results are obtained using the 
so-called method of cumulants (see \cite{SaulStat}) and the local 
dependence structure of nonconventional sums introduced in \cite{book}.
The best exponential inequality obtained by this method yields estimates of the form
\[
P(S_N\geq \ve N)\leq e^{-c(\ve N)^{\frac12}},\, \ve>0,\, N\geq c\ve^{-\frac52}
\]
where $c>0$ is some constant. Such estimates 
are not optimal since the power of $N$ is $\frac12$ and not $1$. 
In the case when $F$ is bounded  we are able to improve these estimates. We first
approximate $S_N$ in the $L^\infty$ norm by martingales with bounded differences
and then apply the Hoeffding-Azuma inequality in order to obtain, in particular, 
estimates of the form
\[
P(S_N\geq \ve N)\leq e^{-c(\ve)N},\,\ve>0,\,N\geq1
\]
where $c(\ve)>0$ is some constant which depends on $\ve$ but not on $N$. In the case when either $\xi_n,\,n\geq 0$ forms a sufficiently fast $\phi$-mixing process or it is generated by a topologically mixing
subshift of finite type or a Young tower with exponential tails we can choose 
$c(\ve)=c\ve^2$ for some $c>0$ which does not depend on $\ve$ and $N$.
Note that all the results described above hold true also with $\bar S_N=S_N-\bbE S_N$ in place of 
$S_N$.
%Results for more general indexes $q_i(n)$'s will be obtained, as well. 
%Because of the varaity of dynamical systems ou apply to, they 
%can be considered as continuation of the research on...Chazzotes book
%Fluctuations of observables in dynamical system...???
%+ Say something about nonlinear large deviations? CHATTERJEE?

Our results hold true, for instance, when
$\xi_n=T^nf$ where $f=(f_1,...,f_d)$, $T$ is a topologically mixing subshift
of finite type, a hyperbolic diffeomorphism (see \cite{Bow}), a Young tower 
(see \cite{Young1} and \cite{Young2}), a Gibbs-Markov map considered in \cite{Jon} or an expanding transformation
taken with a Gibbs invariant measure, as well as in the case when
$\xi_n=f(\Upsilon_n), f=(f_1,...,f_d)$ where $\Upsilon_n$
is a Markov chain satisfying the Doeblin condition considered as a
stationary process with respect to its invariant measure. In fact, 
any stationary and exponentially 
fast $\phi$-mixing sequence $\{\xi_n\}$ can be considered.
In the dynamical systems case each $f_i$ should be either H\" older
continuous or piecewise constant on elements of Markov partitions.
As an application we can consider  $\xi_n=((\xi_n)_1,...,(\xi_n)_\ell)$,
$(\xi_n)_j=\mathds{1}_{A_j}(T^n x)$ in the dynamical systems case and
$(\xi_n)_j=\mathds{1}_{A_j}(\Upsilon_n)$ in the Markov chain case where
$\mathds{1}_{A}$ is the indicator of a set $A$. Let  $F=F(x_1,...,x_\ell)$,
$x_j=(x_j^{(1)},...,x_j^{(\ell)})$ be a bounded H\"older continuous function which identifies with 
the function $G(x_1,...,x_\ell)=x_1^{(1)}\cdot x_2^{(2)}\cdots x_\ell^{(\ell)}$ on the cube $([0,1]^\wp)^\ell$.
Let $N(n)$ be the number of $l$'s
between $0$ and $n$ for which $T^{q_{j}(l)}x\in A_j$ for $j=0,1,...,\ell$
(or $\Upsilon_{q_{j}(l)}\in A_j$ in the Markov chains case), where we set $q_{0}=0$, namely
the number of $\ell-$tuples of return times to
$A_j$'s (either by $T^{q_j(l)}$ or by $\Upsilon_{q_j(l)}$). Then
our results yield moderate deviation theorems and exponential 
concentration inequalities for the numbers $N(n)$. In fact, in this case,  
and more generally for product functions of the form 
$F(x_1,...,x_\ell)=\prod_{i=1}^\ell g_i(x_i)$, our results hold true for 
(stretched) exponentially fast mixing $\al$-mixing processes.
When $f_i$'s and $g_i$'s are H\"older continuous then
our results also hold true for the (deterministic) distance expanding maps considered
in \cite{MSU}, even though there are no underlying Markov partitions.

In general, the sum $S_N$ is a nonlinear function of the random vector $\{\xi_1,\xi_2,...,\xi_{q_\ell(N)}\}$,
and therefore our results can also be viewed as a part of the research on nonlinear large 
deviations theorems (see \cite{Chat1} and \cite{Chat2}). Moreover, in view of the large
variety of dynamical systems that can be considered, our results can be viewed as a part of
the research on concentration of measure for dynamical systems
(see, for instance, \cite{Chaz2}), as well.
%Since  a large variety of dynamical systems can be conisdered, our results also
%contribute to the study of limit theorems and concentarion ineqaulities for dynamical system 
%see \cite{Chaz2}
%Because of the large variety of dynamical systems that can be considered,
%our results can also be veiwed as a part of the research on concentration... 

%\section*{Acknowledgement}
%This paper is a part of the author's PhD thesis conducted at the Hebrew university of Jerusalem. 
%I would like to thank my advisor Professor Yuri Kifer for suggesting to me the problem studied in this 
%paper and for many helpful discussions.

\section{Preliminaries and main results}
\label{sec2}\setcounter{equation}{0}
Our setup consists of a $\wp$-dimensional stochastic process $\xi_n,\,n\geq0$
on a probability space $(\Omega,\cF,P)$ and a family
of sub-$\sig-$algebras $\cF_{k,l}$, $-\infty\leq k\leq l\leq\infty$
such that $\cF_{k,l}\subset\cF_{k',l'}\subset\cF$ if $k'\leq k$
and $l'\geq l$. 
We will impose restrictions on the mixing coefficients
\begin{equation}\label{MixCoef1}
\phi(n)=\sup\{\phi(\cF_{-\infty,k},\cF_{k+n,\infty}): k\in\bbZ\}
\end{equation}
where we recall that for any two sub-$\sigma-$algebras $\cG,\cH\subset\cF$, 
\begin{equation}\label{Phi Gen}
\phi(\cG,\cH)=\sup\Big\{\left|\frac{P(A\cap B)}{P(A)}-P(B)\right|: A\in\cG,\, B\in\cH,\,P(A)>0\Big\}.
\end{equation}

In order to ensure some applications,
in particular, to dynamical systems we will not assume that $\xi_n$
is measurable with respect to $\cF_{n,n}$ but instead impose restrictions
on the approximation rates
\begin{equation}\label{AprxCoef}
\beta_q(r)=\sup_{k\geq0}\|\xi_k-\bbE[\xi_k|\cF_{k-r,k+r}]\|_{q}
\end{equation}
where  $\|X\|_q:=\|X\|_{L^q}$ for any $0< q\leq\infty$ and a random variable
$X$.

We do not require stationarity of the process $\xi_n,\, n\geq 0$, 
assuming only that the distribution of $\xi_n$ does not depend on $n$  and that
the joint distribution of $(\xi_n,\xi_m)$ depends
only on $n-m$, which we write for further
reference by
\begin{equation}\label{2.10}
\xi_n\thicksim\mu\,\,\,\text{ and }\,\, \big(\xi_n,\xi_m\big)\thicksim
\mu_{m-n}
\end{equation}
where $Y\thicksim\mu$ means that $Y$ has $\mu$ for its
distribution. In fact, some of our results hold true assuming only
that $\xi_n\thicksim\mu$ for any $n\geq0$, and we will point out when the 
assumption about the distribution of $(\xi_n,\xi_m)$ is not needed.

Let $F=F(x_1,...,x_\ell)$, $x_j\in\bbR^\wp$
be a function on $(\bbR^\wp)^\ell$ such that for some $K\geq1$, an integer $\lambda\geq 0$,
$\kappa\in(0,1]$ and all $x_i,z_i\in\bbR^\wp$, $i=1,...,\ell$,
we have
\begin{equation}\label{F Hold}
|F(x)-F(z)|\leq K[1+\sum_{i=1}^\ell(|x_i|^\lambda+|z_i|^\lambda)]
\sum_{i=1}^\ell|x_j-z_j|^\kappa
\end{equation}
and
\begin{equation}\label{F Bound}
|F(x)|\leq K[1+\sum_{i=1}^\ell|x_i|^\lambda]
\end{equation}
where $x=(x_1,...,x_\ell)$ and  $z=(z_1,...,z_\ell)$.
In fact, if $\xi_n$ is measurable with respect to $\mathcal F_{n,n}$ then our results will follow with any
Borel function $F$ satisfying (\ref{F Bound}) without imposing (\ref{F Hold}), since the latter is needed only for approximation of $\xi_n$ by conditional expectations $\mathbb E[\xi_n |\mathcal F_{n-r,n+r}]$ using (\ref{AprxCoef}).
To simplify  formulas we assume the centering condition 
\begin{equation}\label{F bar}
\brF :=\int F(x_1,...,x_\ell)d\mu(x_1)\dots d\mu(x_\ell)=0
\end{equation}
which is not really a restriction since
we can always replace $F$ by $F-\brF$. Let $\ell\geq1$ be an integer,
set 
\[
S_N=\sum_{n=1}^NF(\xi_n,\xi_{2n},...,\xi_{\ell n})
\]
and $\bar S_N=S_N-\bbE S_N$. All the results presented here hold true in the 
situation when $q_i(n)$'s are polynomials with positive leading coefficients taking
integer values on the integers, while some of the results hold true even for more
general $q_i(n)$'s. 
This ``nonlinear indexation" case requires some preparation, and so, for the
sake of readability, we will discuss it only in Section \ref{NonLine}.

We will obtain our main results under either
\begin{assumption}\label{ass1}
$\lambda=0$ (i.e. $F$ is a bounded H\"older function)
and there exist $a,d,\eta>0$ so that  
\[
\phi(n)+\beta_\ka^\ka(n)\leq de^{-an^\eta}
\]
for any $n\geq 1$,
\end{assumption}
or
\begin{assumption}\label{ass2} 
$\lambda>0$ and 
there exist $d,a,\eta,M,\zeta>0$ so that 
\[ 
\phi(n)+\beta_\infty^\ka(n)\leq de^{-an^\eta}
\] 
for any $n\geq 1$, and for any $k\in\bbN$,
\begin{equation*}\label{Moment growth assumptions}
\tau_k^k=\bbE|\xi_1|^k=\int |x|^kd\mu(x)\leq M^k(k!)^\zeta.
\end{equation*}
\end{assumption}
%Note that (\ref{Moment growth assumptions})  is equivalent to $\bbE[e^...]<\infty$ (see page 5 in 
%\cite{Dor}).
Note that under either Assumption \ref{ass1} or Assumption \ref{ass2} there exists a constant $a_0$ 
so that $|\bbE S_N|\leq  a_0K$ for any $N\geq1$. In fact, this estimate holds true under weaker conditions, see the paragraph proceeding Theorem \ref{D-thm}.

Our first result is the following 
\begin{theorem}\label{Thm:ModDevNonc}
(i) Suppose that Assumption \ref{ass1} holds true and set $\gam=\frac1\eta$. Then there exist constants $c_1,c_2>0$ which depend only on $K,\ell,d,a,\eta$ and $\ka$ so that for any $x>0$,
 \begin{equation}\label{FirstExpCon}
P(\bar S_N\geq x)\leq \exp\Big(-\frac{x^2}{2(c_1+c_2xN^{-\frac1{2+4\gam}})^{\frac{1+2\gam}{1+\gam}}}\Big).
\end{equation}

(ii) When Assumption \ref{ass2} holds true then (\ref{FirstExpCon})
hold true with $\gam=\frac1\eta+\lambda\zeta$ in place of $\frac1\eta$ and 
constants $c_1$ and $c_2$ which depend only on 
$K,\ell,d,a,\eta,M,\zeta,\ka,\lambda$ and $\tau_\lambda$.
\end{theorem}

The above theorem holds true also for certain nonlinear $q_i(n)$'s such as polynomials and functions with exponential growth, see Section \ref{NonLine}.
Note that when $\be_q(r_0)=0$ for some $q$ and $r_0$ then Theorems \ref{Thm:ModDevNonc} hold
true for any Borel function $F$ satisfying (\ref{F Bound}), namely, there is no need of (\ref{F Hold})
or of any other type of continuity.

Next, by taking $x=\ve N$, $\ve>0$ in (\ref{FirstExpCon})  (or in the corresponding 
estimate under Assumption \ref{ass2}) and using that $|\bbE S_N|\leq a_0 K$ we obtain that 
\begin{equation}\label{App aN}
\max\big(P(\bar S_N\geq \ve N),P(S_N\geq \ve N)\big)
\leq e^{-c_7(\ve N)^{\frac 1{1+\gam}}},\,\,N\geq c_6\ve^{-2-\frac1\gam} 
\end{equation}
where $c_6$ and $c_7$ are positive constants which do not depend on $N$ and $a$, and $\gam$ equals 
either $\frac1\eta$ or $\frac1\eta+\lambda\zeta$, depending on the case.
The power of $N$ in (\ref{App aN}) is not optimal since it is smaller than $1$. 
In order to obtain more accurate estimates on the tail probabilities we also prove the following
\begin{theorem}\label{MartExp-Cor}
Suppose that $F$ is a bounded H\"older continuous function
and that 
\[
\varphi:=\sum_{n=0}^\infty\phi(n)<\infty. 
\]
Fix some $N\geq1$ and $r\geq0$ and set 
$\del_1:=K(\varphi+r+1)$  and $
\del_2=KN\be_\infty^\ka(r)+\del_1$.
Then there exists
a constant $B>0$ which depends only on $\ell$ so that for
any $\la>0$,
\begin{equation}\label{ExpEstSN}
\bbE e^{\la S_N}\leq e^{B\la^2N\ell\del_1+B\la\del_2}.
\end{equation}
When $\be_\infty(r_0)=0$ for some $r_0\geq0$ then the above results hold true with $r=r_0$
for any bounded Borel function $F$, i.e. there is no need of any kind of continuity.
\end{theorem}
Theorem \ref{MartExp-Cor} holds true also when $q_i(n)$'s are polynomials with positive leading coefficients taking integer values on the integers, see Section \ref{NonLine}.
Note that the above theorem does not require that $(\xi_n,\xi_m)\thicksim
\mu_{m-n}$ since it does not involve the limit $D^2$ (which does not necessarily exist without
this assumption about the distribution of $(\xi_n,\xi_m)$).

Next, using the Chernoff bounding method, in Section \ref{sec4} we derive from (\ref{ExpEstSN}) that 
for any $t>0$,
\begin{equation}\label{ExpConc}
P(S_N\geq t+B\del_2)\leq e^{-\frac{t^2}{4B^2N\ell \del_1^2}}
\end{equation}
When $\be_\infty(r_0)=0$ for some $r_0\geq0$ then by taking $r=r_0$ the terms $\del_1$ and 
$\del_2$ are constants, and therefore we obtain optimal exponential concentration
inequalities of the form 
\[
P(S_N\geq \ve N)\leq e^{-c\ve^2N},\,\,N\geq\frac{2B\del_2}{\ve}
\] 
where $c=\frac{\del_2}{16\ell\del_1^2}>0$ and $\ve>0$.
When $\be_\infty(r)$ convergence to $0$ as $r\to\infty$ then for any $\ve>0$ we can take a sufficiently large
$r_0=r_0(\ve)$ and obtain that there exists a constant $c(\ve)>0$ so that for any $N\geq1$ and $t>0$,
\[
P(S_N\geq t+0.5\ve N)\leq e^{-c(\ve)\frac{t^2}{N}}
\] 
and in particular
\begin{equation}\label{Above}
P(S_N\geq \ve N)\leq e^{-c_1(\ve)N}
\end{equation}
for some constant $c_1(\ve)>0$ which depends on $\ve$ but not on $N$.
When some rate of decay of $\beta_\infty^\ka(r)$ to $0$ is known we can find an explicit $c(\ve)$. 
For instance, when $\beta_\infty^\ka(r)\leq de^{-ur},\, d,u>0$ for any $r\geq0$, 
we can take $r_0$ of the form $r_0=-c\ln\ve$ and then the above estimate will hold true 
with $c(\ve)$ having the form $c(\ve)=q_0|\ln\ve|^{-1}$ for some constant $q_0$
which depends only on $\ell,d,u,\ka$ and $K$. 
%Say something about this...
\begin{remark} 
Let $(\cX,T)$ be a Young tower (see \cite{Young1} and \cite{Young2}) and $\mu$
be an appropriate Gibbs measure. 
Consider $\sig$-algebras $\cF_{n,m}$ are generated by an appropriate 
Markov partition. Then (see \cite{Hyd1}), the mixing coefficients $\phi(n)$ decay in the same speed
as the tails of the tower. Let $h_1,...,h_\wp$ be real valued functions 
on $\cX$ which are either constant
on atoms of the partition or are H\"older continuous functions
and let $\xi_n=(h_1\circ T^n,...,h_\wp\circ T^n),\,n\geq1$. 
Then, (\ref{Above}) holds true (with an appropriate $c(\ve)$'s) assuming 
that the tails converge sufficiently fast to $0$. Note that when 
 $h_1,...,h_\wp$ are H\"older continuous functions, then the centralized sum $\bar S_N$ can be written as a reverse martingale, and therefore (see \cite{Chaz}), in these circumstances
we obtain optimal exponential concentration inequality of the form
\[
P(\bar S_N\geq t)\leq e^{-c\frac{t^2}{N}},\,N\geq1,\,t>0
\]  
where $c$ is some constant. Plugging in $t=\ve N,\,\ve>0$ we 
derive that for any $N\geq 1$,
\[
P(\bar S_N\geq \ve N)\leq e^{-c\ve^2N},
\]
namely we can take $c(\ve)$ of the form $c(\ve)=c\ve^2$ when $S_N$ is replaced with $\bar S_N$.
\end{remark}

Recall now (see \cite{DemZet}) that a sequence of probability measures $\mu_N,\,N\geq1$ on a topological space 
$\cX$ is said to satisfy the large deviation
principle (LDP) with speed $s_N\nearrow\infty$ and good rate function $I(\cdot)$
if $I$ is lower semicontinuous, the sets $I^{-1}[0,\al],\,\al\geq0$ are compact and for 
any Borel measurable set $\Gam\subset\cX$,
\[
\liminf_{N\to\infty}\frac1{s_N}\ln\mu_N(\Gam)\geq-\inf_{x\in\Gam^o}I(x)
\]
and 
\[
\limsup_{N\to\infty}\frac1{s_N}\ln\mu_N(\Gam)\leq-\inf_{x\in\bar\Gam}I(x)
\]
where $\Gam^o$ denotes the interior of a set $\Gam$ and $\bar\Gam$ denotes its closure.
A sequence of random variables $W_N, N\geq 1$ is said to satisfy the LDP with speed $s_N$ and
good rate function $I(\cdot)$ if the sequence $\cL(W_N),\,N\geq 1$ of the laws of the $W_N$'s satisfies the appropriate LDP.
We also recall the following terminological convention. 
When $W_N,\,N\geq 1$ satisfies the law of large numbers  and $s_N$ grows slower than linear in 
$N$ the appropriate LDP is usually called a moderate deviation principle (MDP) and the case when $s_N=N$
is referred to as the LDP.

We will also prove the following

\begin{theorem}\label{Thm2.4}
(i) Suppose that Assumption \ref{ass1} holds true and set $\gam=\frac1\eta$.
Set 
$v_N=\sqrt{\mathrm{Var}(S_N)}$ and when $v_N>0$ also set
$Z_N=\frac{\bar S_N}{v_N}$.
Let $\Phi$ be the standard normal distribution function. Then
the limit
$D^2=\lim_{N\to\infty}\frac1N\bbE S_N^2$ 
exists and 
when $D^2>0$ there exist 
constants $c_3,c_4,c_5>0$ which depend only on 
$\ell,K,\ka,a,d$ and $\eta$
so that for any $N\geq c_3$ we have $v_N>0$ and for any $0\leq x<c_4 N^{\frac1{2+4\gam}}$,
\begin{eqnarray}\label{ModDev3}
\left|\ln\frac{P(Z_N\geq x)}{1-\Phi(x)}\right|\leq c_5(1+x^3)
N^{-\frac1{2+4\gam}}\,\,\text{ and}\\
\left|\ln\frac{P(Z_N\leq -x)}{\Phi(-x)}\right|\leq c_5(1+x^3)
N^{-\frac1{2+4\gam}}.\nonumber
\end{eqnarray}
Moreover, let $a_N,\,N\geq 1$ be a sequence of real numbers so that 
\[
\lim_{N\to\infty}a_N=\infty\,\,\text{ and }\,\,\lim_{N\to\infty}{a_N}{N^{-\frac{1}{2+4\gam}}}=0.
\]
Then the sequence  $(DN^{\frac12}a_N)^{-1}S_N,\,N\geq1$ satisfies the MDP
with  the speed $s_N=a_N^2$ and the rate function $I(x)=\frac{x^2}2$.

(ii) When Assumption \ref{ass2} holds true all the results stated
above hold true with $\gam=\frac1\eta+\lambda\zeta$ in place of $\frac1\eta$ and 
constants $c_1,c_2$ and $c_3$ which depend only on 
$K,\ell,d,a,\eta,M,\zeta,\ka,\lambda$ and $\tau_\lambda$.
\end{theorem}

Theorem \ref{Thm2.4} also holds true when $q_i(n)$'s are polynomials, or functions with certain exponential growth,
see Section \ref{NonLine}. When $\be_q(r_0)=0$ for some $q$ and $r_0$ then all the results stated in Theorem \ref{Thm2.4} hold
true for any Borel function $F$ satisfying (\ref{F Bound}).
We also remark that (\ref{ModDev3}) is obtained using Lemma 2.3 
in \cite{SaulStat}. This lemma yields certain estimates close to the ones in \ref{ModDev3}, but for larger
domain of $x$'s. For the sake of readability these results are not stated here.

Theorems \ref{Thm:ModDevNonc}, \ref{MartExp-Cor} and \ref{Thm2.4} will follow from the following general results. 
The first one is
\begin{theorem}\label{D-thm}
Suppose that for some $b>2$ and $m>0$, 
\begin{equation}\label{Mom}
\frac1b\geq\frac\lambda m+1,\,\,\max(\tau_m,\tau_{\lambda b})<\infty
\end{equation}
and
\[
\Theta(b,\ka):=\sum_{n=0}^\infty(n+1)\phi^{1-\frac1b}(n)+
\sum_{n=0}^\infty(n+1)\be_\ka^\ka(n)<\infty.
\] 
Then the  limit $D^2=\lim_{N\to\infty}\frac1N\bbE S_N^2$ exists and there exists 
$c_\ell>0$ which depends only on $\ell$
so that
\begin{equation}\label{D-Rate}
\big|\bbE S_N^2-D^2N\big|\leq c_\ell C_0N^{\frac12}
\end{equation}
for any $N\in\bbN$, where $C_0=K^2(1+\gam_{m}^\lambda)\Theta(b,\ka)$. 
Moreover,
$D^2>0$  if and only if there exists no stationary in the wide sense process $\{V_n: n\geq 1\}$
such that 
\[
F(\xi^{(1)}_n,\xi^{(n)}_{2n},...,\xi_{\ell n}^{(\ell)})=V_{n+1}-V_n,\,P-\text{a.s.}
\]
for any $n\in\bbN$, where $\xi^{(i)}$, $i=1,...,\ell$ are independent copies of $\xi=\{\xi_n: n\geq1\}$. When $\lambda=0$ then the above results hold true without assuming (\ref{Mom}) while when $\be_\infty(r)=0$ for some $r$ they hold true for Borel measurable $F$'s without assuming (\ref{F Hold}).
\end{theorem}
This theorem is a particular case of Theorem 1.3.4 in \cite{book} and Theorem 2.2 in \cite{Ha}. 
In fact, it is a consequence of the arguments in
\cite{KV}, \cite{Ki3} and  \cite{HK1} and is formulated here for 
readers' convenience.
We refer the readers to \cite{HKllt} for conditions in the special case when $\xi_n,\,n\geq 0$
forms a sufficiently fast mixing Markov chain.
Remark that in the circumstance of Theorem \ref{D-thm} there exists a constant $a_\ell$ which 
depends only on $\ell$ so that $|\bbE S_N|\leq  a_\ell KC_0$ for any $N\geq 1$. Indeed this is a consequence of (\ref{F bar}) and
Corollary 1.3.14 in \cite{book}. Therefore, for any $N\geq1$,
\begin{equation}\label{VarEst}
\big|\mathrm{Var}(S_N)-D^2N\big|\leq C_1N^{\frac12}
\end{equation}
for some constant $C_1$ which depends only on $C_0,\ell$ and $K$.

We recall next that the $k$-th cumulant of a random variable $W$ with finite moments of all 
orders is given by
\[
\Gam_k(W)=\frac1{i^k}\frac{d^k}{dt^k}\big(\ln\bbE e^{itW}\big)\big|_{t=0}.
\]
Note that $\Gam_1(W)=\bbE W$, $\Gam_2(W)=\mathrm{Var}(W)$ and that 
$\Gam_k(aW)=a^k\Gam_k(W)$ for any $a\in\bbR$ and $k\geq 1$.
\begin{theorem}\label{NoncCum}
Under Assumption \ref{ass1} there exists a constant $c_0$ which depends only on 
$K,\ell,d,a,\eta$ and $\ka$ so that for any $k\geq3$,
\[
|\Gam_k(\bar S_N)|\leq N(k!)^{1+\gam_1}(c_0)^{k-2}
\]
where $\gam_1=\frac{1}\eta$. When
Assumption \ref{ass2} holds true there exists a constant $c_0$ which depends only on
$K,\ell,d,a,\eta,M,\zeta,\ka$ and $\lambda$
so that for any $k\geq3$,
\[
|\Gam_k(\bar S_N)|\leq N(k!)^{1+\gam_2}(c_0)^{k-2}
\]
where $\gam_2=\gam_1+\lambda\zeta$.
\end{theorem}
Note that Theorem \ref{NoncCum} holds true without assuming that $(\xi_n,\xi_m)\thicksim
\mu_{m-n}$ since its proof does not require that the limit $D^2$ exists. When 
$(\xi_n,\xi_m)\thicksim\mu_{m-n}$ then $N^{-\frac12}\bar S_N$ satisfies the CLT
and so the term $N$ on the above right hand sides should not be alarming 
since theorem \ref{NoncCum} implies that 
\[
|\Gam_k(N^{-\frac12}\bar S_N)|\leq (k!)^{1+\gam}(N^{-\frac12}c_0)^{k-2}
\]
for any $k\geq3$, where $\gam$ is either $\gam_1$ or $\gam_2$, depending on the case.  
After establishing Theorem \ref{NoncCum} the moderate deviations theorems and (stretched) exponential concentration 
inequalities stated in Theorems \ref{Thm:ModDevNonc} and \ref{Thm2.4} follow from the so called method of cumulants (see \cite{SaulStat} and \cite{Dor}).

Theorem \ref{MartExp-Cor} will follow from the following result together with the  Hoeffding-Azuma inequality. 

\begin{theorem} \label{MartExp}
Suppose that $F$ is a bounded H\"older function
and that 
\[
\varphi:=\sum_{n=0}^\infty\phi(n)<\infty. 
\]
 Then there exists a constant
$B>0$ which depends only on $\ell$ so that for any $N\geq1$ and $r\geq 0$ there is a 
martingale $M^{(N,r)}_{n},\,n\geq 1$ whose differences are bounded by 
$\del_1':=BK(\varphi+r+1)$ and 
\[
\|S_N-M_{\ell N}^{(N,r)}\|_{\infty}\leq \del_2':=BKN\be_\infty^\ka(r)+\del_1'.
\]
When $\be_\infty(r_0)=0$ for some $r_0\geq0$ then the above results hold true with $r=r_0$
for any bounded Borel  function $F$.
\end{theorem}
 
%We also not that  when taking $\ve=\ve_N$ of logarithmic
%order we obtain estimates of the form
%\[
%P(S_N\geq t+\ln N)\leq e^{-c\frac{t^2}{N\ln N}},\,N\geq 1,\,t>0
%\]
%and so for any $a$ and $N$ such that $2\ln N\leq aN$,
%\[
%P(S_N\geq aN)\leq e^{-4c\frac{a^2N}{\ln N}}.
%\]

\subsection{Product functions}\label{PrFun0}
In the special case when $F$ has the form 
\begin{equation}\label{Product form}
F(x_1,...,x_\ell)=\prod_{i=1}^\ell f_i(x_i)
\end{equation}
the results stated in Theorems \ref{Thm:ModDevNonc} (i) and Theorem \ref{Thm2.4} (i) hold true under weaker assumptions,
as described in what follows. %This ``product function" case is the one considered in 
%ergodic theory

Recall first that the $\al$-mixing coefficients are given by 
\begin{equation}\label{alCoeff}
\al(n)=\sup\{\al(\cF_{-\infty,k},\cF_{k+n,\infty}): k\in\bbZ\}
\end{equation}
where  for any two sub-$\sigma-$algebras $\cG,\cH\subset\cF$,
\begin{equation}\label{al def}
\al(\cG,\cH)=\sup\big\{\big|P(A\cap B)-P(A)P(B)\big|: A\in\cG,\, B\in\cH\big\}.
\end{equation}
Then (see \cite{Br}) $\al(n)\leq\frac12\phi(n)$ for any $n\geq0$, and so, assumptions involving $\al(n)$ are weaker than ones involving $\phi(n)$. We also recall that (see \cite{Douk0}) for any 
bounded functions $g_1,...,g_L$, numbers $m_1<n_1<m_2<n_2<...<m_L<n_L$ and $\cF_{m_i,n_i}$-measurable random vectors $U_i$, $i=1,2,...,L$,
 \begin{equation}\label{Aapha block result0}
\left|\bbE \prod_{i=1}^Lg_i(U_i)-\prod_{i=1}^L\bbE g(U_i)\right|\leq 8\big(\prod_{i=1}^L\sup|g_j|\big)
\sum_{t=2}^{L}\al(m_t-n_{t-1}).
\end{equation}
Relying on (\ref{Aapha block result0}) we show in Section \ref{ProdCase} that all the results stated in Theorems \ref{Thm:ModDevNonc} (i)  and Theorem \ref{Thm2.4} (i) hold true when $f_i$'s are bounded. The situation of unbounded $f_i$'s satisfying certain moment conditions is discussed there, as well.

Next, let $T:\Om\to\Om$ be a measurable and $P$-preserving map. We assume here that there exists 
a space $\cH$ of real valued bounded functions 
on $\Om$, a norm $\|\cdot\|_\cH$ on $\cH$, a constant $d$ and a sequence $c(m),\,m\geq 1$, 
which converges to $0$ as $m\to\infty$,
 so that for any $f,g\in\cH$ and $n\geq 1$,
\begin{equation}\label{DecCorCond}
\text{Cor}_P(g,f\circ T^n)\leq d\|g\|_\cH\sup|f|c(n).
\end{equation}
Usually $\Om$ will be a topological space and $\cH$ will be a space of H\"older continuous 
functions equipped with an appropriate norm. 
We also assume that the $f_i$'s are members of $\cH$.
Obtaining the MDP and exponential concentration inequalities under condition (\ref{DecCorCond})
is important when either there are no underlying Markov partitions
or there is no effective estimate on the diameter of such partitions (so it is impossible to
approximate effectively H\"older continuous functions by functions which are constant on elements of 
such partitions). For instance, (\ref{DecCorCond}) holds true with $c(n)=e^{-an},\,a>0$
in the (nonrandom) setup of \cite{MSU},
where $T$ is a locally distance expanding map and $\cH$ is a space of (locally) H\"older continuous functions, 
while there are no underlying Markov partitions.
Let  $n_1<n_L<...<n_L$ and $g_1,...,g_L\in\cH$. By writing 
\[
\prod_{i=1}^{L}g_i\circ T^{n_i}=\big(g_1\cdot G\circ T^{n_2-n_1}\big)\circ T^{n_1}
\]
where $G=\prod_{i=2}^L g_i\circ T^{n_i-n_2}$ we obtain that
\begin{equation}\label{DecCor block result0}
\left|\bbE_P\prod_{i=1}^L g_i\circ T^{n_i}-\prod_{i=1}^L\bbE_Pg_i\circ T^{n_i}
\right|\leq dM^L\sum_{t=2}^L c(n_t-n_{t-1})
\end{equation}
where $M=\max\{\sup|g_i|,\|g_i\|_\cH:\,i=1,2,...,L\}$.
Suppose next that 
\[
\sum_{n=1}^\infty nc(n)<\infty.
\]  
Using (\ref{DecCor block result0}) in place of (\ref{Aapha block result0}), we will prove in Section \ref{ProdCase} that all the results stated in Theorem \ref{Thm:ModDevNonc} (i), Theorem \ref{Thm2.4} (ii)
and  Theorem \ref{D-thm} hold true with $\be_\ka(n)\equiv0$ and $c(n)$ in place of $\phi(n)$.

\section{Nonconventional moderate deviations and exponential inequalities via the method of cumulants}
\label{sec3}\setcounter{equation}{0}
\subsection{General estimates of cumulants} 
%In this section we collect some known results...Do we really want to cite these results here?
%First cite some general results which follow from estimates
%on cumulants from  \cite{SaulStat} and \cite{Dor}.

%\begin{theorem}
%Theorem 1.1 in \cite{Dor} 
%\end{theorem}

%\begin{lemma}
%Lemmas 2.3 and 2.4 in \cite{SaulStat} (see also Lemmas 6.1 and 6.2 in \cite{Dor}) 
%\end{lemma}

Let $V$ be a finite set and $\rho:V\times V\to[0,\infty)$ be so that $\rho(v,v)=0$
and $\rho(u,v)=\rho(v,u)$ for any $u,v\in V$. For any $A,B\subset V$ set 
\[
\rho(A,B)=\min\{\rho(a,b): a\in A, b\in B\}.
\] 
Let $X_v,\, v\in V$ be a collection of centered random variables 
with finite moments of all orders, and for each $v\in V$ and $t\in(0,\infty]$ let 
$\varrho_{v,t}\in(0,\infty]$ be so that $\|X_v\|_t\leq\varrho_{v,t}$.  
Set $W=\sum_{v\in V} X_v$.  
The following result is (essentially) proved in \cite{Gorc} (see Theorem 1 there).

\begin{theorem}\label{Gorc}
Let $0<\del\leq\infty$.
Suppose that for any $k\geq 1$, $b>0$
and a finite collection $A_j,\,j\in\cJ$ of (nonempty) subsets of $V$ so that 
$\min_{i\not=j}\rho(A_i,A_j)\geq b$ and $r:=\sum_{j\in\cJ}|A_j|\leq k$ we have
\begin{equation}\label{MixGorc}
\left|\bbE\prod_{j\in\cJ}\prod_{i\in A_j}X_i-\prod_{j\in\cJ} \bbE\prod_{j\in A_j}X_i\right|\leq
(r-1)\Big(\prod_{j\in\cJ}\prod_{i\in A_j}\varrho_{i,(1+\del)k}\Big)\gam_\del(b,k)
\end{equation}
where $\gam_\del(b,r)$ is some nonnegative number which depends only
on $\del,b$ and $r$, and $|\Del|$ stands for the cardinality of a finite set $\Del$.
Then for any $k\geq2$ and $s>0$, 
\[
|\Gam_k(W)|\leq k^k\Big(2^kC(k)(L_s(k))^{k-1}+R_s(\del,k)\Big)
\]
where for any $0<t\leq\infty$,
\begin{eqnarray*}
L_s(t)=\sup\big\{\sum_{u\in V: \rho(u,v)\leq s}\varrho_{u,t}:\,v\in V\big\},\,\,
C(t)=\sum_{v\in V}\varrho_{v,t},\\
R_s(\del,k)=\sum_{m\geq s+1}\big(L_m((1+\del)k)\big)^{k-1}C((1+\del)k)\la(\tilde\gam_\del(m,k),k),\\
\tilde\gam_\del(m,k)=\max\{\gam_\del(m,r)/r: 1\leq r\leq k\}\\
\text{and }\,\,\la(\ve,k)=k!\sum_{r=1}^{[\frac k2]}\frac{\ve^r(3r+1)^{k-2r}}{r(k-2r)!}.
\end{eqnarray*} 
\end{theorem}
The difference in the formulations of Theorem 1 in \cite{Gorc} and Theorem \ref{Gorc} is
that the result from \cite{Gorc} relies on a certain local mixing condition instead of (\ref{MixGorc}). 
But in  proof from there the author obtains (\ref{MixGorc}) with $\varrho_{v,t}=\|X_v\|_t$ and appropriate
$\gam_\del(b,k)$ relying on that 
mixing condition, and so Theorem \ref{Gorc} is proved exactly as in \cite{Gorc}. We reformulated this theorem
in order to include the case when $\beta_{q}(r)\not=0$ for any $r$ and the second situation considered in
Section \ref{PrFun0}.

% author used some (local) mixing coefficient in order to show that 
%(\ref{MixGorc}) holds. We needed to formulate his result without mixing coefficients 
%since in some of our applications we will not have such coefficients.
Note that by Stirling's approximation there exists a constant $C>0$ so that 
$k^k\leq Ce^kk!$ for any $k\geq1$. Remark also that when condition (\ref{MixGorc}) holds true 
only in the case when $|\cJ|=2$, then using induction this implies that (\ref{MixGorc}) holds true with 
$k\gam_\del(b,k)$ instead of $\gam_\del(b,k)$, for collections of more than two sets. 
Compare this with \cite{KalNe}, \cite{Douk} and \cite{Ded} in the case when $V=\{1,...,n\}$ and $\rho(x,y)=|x-y|$.

Next, the following result is a consequence of Theorem \ref{Gorc}.

\begin{corollary}\label{GorcCor}
Suppose, in addition to the assumptions of Theorem \ref{Gorc}, 
that there exist $c_0\geq1$ and $u_0\geq 0$ so that 
\begin{equation}\label{linear rho}
|\{u\in V: \rho(u,v)\leq s\}|\leq c_0s^{u_0}
\end{equation}
for any $v\in V$ and $s\geq 1$. Assume also that $\tilde\gam_\del(m,k)\leq de^{-am^\eta}$ for some
$a,\eta>0$, $d\geq 1$ and all $k,m\geq1$. Then there exists a constant $c$ which depends only
on $c_0,a,u_0$ and $\eta$ so that for any $k\geq 2$, 
\begin{equation}\label{GorEst1}
|\Gam_k(W)|\leq d^k|V|c^k(k!)^{1+\frac{u_0}\eta}\big(M_k^k+M_{(1+\del)k}^k\big)
\end{equation}
where for any $q>0$,
\[
M_{q}=\max\{\varrho_{v,q}:\,v\in V\}\,\,\text{ and }\,\,M_q^k=(M_q)^k.
\]
When the $X_v$'s are bounded and (\ref{MixGorc}) holds true with $\del=\infty$
we can always take $\varrho_{v,t}=\varrho_{v,\infty},\, t>0$ and 
then for any $k\geq2$,
\begin{equation}\label{GorEst1-cor}
|\Gam_k(W)|\leq 2d^k|V|M_\infty^kc^k(k!)^{1+\frac{u_0}\eta}.
\end{equation}
When $\del<\infty$ and there exist $\te\geq0$ and $M>0$ so that
\begin{equation}\label{MomGorch}
(\varrho_{v,k})^k\leq M^k(k!)^{\te}
\end{equation}
for any $v\in V$ and $k\geq1$, then for any $k\geq2$,
\begin{equation}\label{GorEst2}
|\Gam_k(W)|\leq 3C^{\frac{\te}{1+\del}}d^k|V|c^k(1+\del)^k
M^k(k!)^{1+\frac{u_0}{\eta}+\te}
\end{equation}
where $C$ is some absolute constant. 
\end{corollary}
The proof of this corollary is elementary but for readers' convenience we 
will give all the details.

\begin{proof}
Let $k\geq2$ and $m\geq s\geq k^{\frac1\eta}$. Set $\ve=\ve_m= e^{-am^\eta}$. Then $\tilde\gam_\del(m,k)\leq d\ve$
and so
\begin{equation*}
\la(\tilde\gam_\del(m,k),k)\leq d^k k!4^k\sum_{r=1}^{[\frac k2]}\frac{\ve^rr^{k-2r-1}}{(k-2r)!}\leq
d^kk!4^k\sum_{r=1}^{[\frac k2]-1}\frac{\ve^rr^{k-2r}}{(k-2r)!}+d^kk!4^k\ve^{[\frac k2]}.
\end{equation*}
Observe that $k!4^k\ve^{[\frac k2]}\leq H\ve$ for some constant $H$ which depends only 
on $a$ and $\eta$, where we used that $m^\eta\geq k$. Moreover,
by Stirling's approximation there exists an absolute constant $C>0$ so that for any $1\leq r\leq[\frac k2]-1$, 
\[
\frac{1}{(k-2r)!}\leq C\frac{e^{k-2r}}{(k-2r)^{k-2r}}.
\]
Therefore,
\begin{equation}\label{basic lam estim0}
\la(\tilde\gam_\del(m,k),k)\leq Ck!(4de)^k\sum_{r=1}^{[\frac k2]-1}\ve^r\big(\frac{r}{k-2r}\big)^{k-2r}+
d^kH\ve.
\end{equation} 
Consider next the function $g_m=g_{m,k}:[1,\frac k2-1]\to\bbR$ given by 
\[
g_m(r)=\ve^r\big(\frac{r}{k-2r}\big)^{k-2r}=e^{r\ln\ve-(k-2r)\ln(\frac kr-2)}.
\]
Then, 
\[
g_m'(r)=\big(\ln\ve+2\ln(\frac kr-2)+\frac kr\big)g_m(r).
\]
If $g_m'(r_0)=0$ for some $r_0\in[1,\frac k2-1]$ then 
\[
ak\leq am^\eta=-\ln\ve=2\ln(\frac k{r_0}-2)+\frac k{r_0}\leq \frac{3k}{r_0}
\]
and so $r_0\leq\frac{3}{a}:=q$. Hence,
\[
\max_{r\in[1,\frac k2-1]}g_m(r)\leq\max\big(g_m(1),g_m(\frac k2-1),\max_{w\in[1,q_k]}g_m(w)\big)
\] 
where $q_k=\min\big(\frac k2-1,q\big)$ and we set $\max\emptyset=-\infty$.
Observe now that 
\[
g_m(1)=\frac{\ve}{(k-2)^{k-2}}\leq\frac{k^2\ve}{k!}\leq 3^k (k!)^{-1}\ve.
\]
Since $m^\eta\geq k$ we also have 
\[
g_m(\frac k2-1)\leq k^2\ve^{\frac k2-1}\leq \ve k^2e^{-ak(\frac k2-2)}
\leq c_1\ve(k!)^{-1}
\]
where $c_1$ is a constant which depends only on $a$.
When $k\leq 2(q+1)$ we can trivially write
\[
\max_{w\in[1,q_k]}g_m(w)\leq\ve(\psi_0)^k(k!)^{-1}
\] 
for some constant $\psi_0$ which depends only on $a$. On the other hand, when $k>2(q+1)$ then
using that the function $x\to x^{-x}$ 
is strictly decreasing on $[1,\infty)$ and then Stirling's approximation we derive that
\[
\max_{w\in[1,q_k]}g_m(w)=\max_{w\in[1,q]}g_m(w)\leq \ve (q+1)^k(k-[2q]-1)^{-(k-[2q]-1)}\leq\ve\psi^k(k!)^{-1}
\]
where $\psi$ is a constant which depends only on $a$, and we also used the inequality 
$k!\leq (k-l)!k^l\leq(k-l)!3^{kl},\,1\leq l\leq k$. 
We conclude from the above estimates that there exists a constant $R=R(a,\eta)$ which depends only on $a$ and $\eta$ so that
for any $1\leq r\leq \frac k2-1$,
\[
g_m(r)=\ve^r\big(\frac{r}{k-2r}\big)^{k-2r}\leq \ve R^k(k!)^{-1}
\]
which together with (\ref{basic lam estim0}) yields
\begin{equation}\label{basic lam estim}
\la(\tilde\gam_\del(m,k),k)\leq d^kR_0^k\ve=d^kR_0^ke^{-am^\eta}
\end{equation}
where $R_0=R_0(a,\eta)\geq 1$ is another constant.

%Prove in the case when we have stretched exponential $e^{-ax^\alpha}$ for some $a,\alpha>0$? 
%this is important since perhaps in several Young tower one may be able to approximate 
%the first hitting times tail probabilities by another process which only satisfies certain
%moderate deviations which will lead to stretched exponential mixing: The order of the 
%expressions changes, so I'll remark about it and then talk about stretched exponential
%rates in the "extensions"
Next, using (\ref{linear rho}), (\ref{basic lam estim}) and the definitions of 
$C(t)$ and $L_s(t)$ we obtain that
\[
R_s(\del,k)\leq d^k(1+H)R_0^k(M_{(1+\del)k})^k|V|\sum_{m\geq s+1}m^{u_0(k-1)}e^{-am^\eta}
\]
where $L_s(t),C(t),R_s(\del,k)$ are defined in 
Theorem \ref{Gorc}. Set $j_0=j_0(k,\eta)=[\frac{(k-1)u_0+2}{\eta}]+1$. 
Then 
\[
m^{u_0(k-1)}e^{-am^\eta}\leq m^{u_0(k-1)}j_0!(am^\eta)^{-j_0}\leq j_0!a^{-j_0}m^{-2} .
\] 
By Stirling's approximation there exists a constant $Q$ which depends only on $\eta$ and $u_0$ so that 
$j_0!\leq Q^k(k!)^{\frac{u_0}\eta}$ and therefore,
\[
\sum_{m\geq s+1}m^{u_0(k-1)}e^{-am^\eta}\leq j_0!\sum_{m\geq s+1}\frac1{m^2}\leq 
\frac{1}sj_0!\leq\frac{1}s(Q_1)^k(k!)^{\frac{u_0}\eta}  
\]
where $Q_1$ is a constant which depends only on $\eta,a$ and $u_0$.
%Since $\gam_\del(m,r)\leq dre^{-am}$ we obtain that 
%\[
%R_s(\del,k)\leq (M_{(1+\del)k})^kd^kk!|V|\sum_{m\geq s+1}m^{k-1}e^{-am}.
%\]
%Next, for any $t\geq1$ and $a>0$,
%\[
%\sum_{m\geq t}m^ke^{-am}\leq e^a\sum_{m\geq t}\int_{m}^{m+1}x^ke^{-ax}dx
%=\frac{e^a}{a^{k+1}}\int_{at}^{\infty}y^ke^{-y}dy.
%\]
%Set $w_k(q)=\int_{q}^{\infty}y^ke^{-y}dy$. Integration by parts shows that
%\[
%w_k(q)=q^ke^{-1}+kw_{k-1}(q).
%\]
%Iterating this recursion rule we obtain that 
%\[
%w_k(q)=k!e^{-q}q^{k}\sum_{j=0}^k\frac{q^{-j}}{j!}\leq
%k!q^ke^{-q+q^{-1}}.
%\]
%Substituting $q=at$ we obtain that $w_k(at)\leq e^{a^{-1}}k!a^kt^ke^{-at}$ 
%and therefore
%\[
%\sum_{m\geq t}m^ke^{-am}\leq a^{-1}e^{a^{-1}}t^ke^{-at}.
%\]
%We conclude that
%\begin{eqnarray*}
%R_s(\del,k)\leq g(a)(M_{(1+\del)k})^kd^k(k!)^2|V|(s+1)^ke^{-a(s+1)}\\\leq
%g(a)2^ka^{-k}(M_{(1+\del)k})^kd^k(k!)^2|V|(as)^ke^{-as}
%\end{eqnarray*}
%where $g(a)=a^{-1}e^{a^{-1}}$, 
Taking $s=k^\frac1\eta$ the estimate (\ref{GorEst1})
follows by Theorem \ref{Gorc}, the definition of $L_s(m)$, Stirling's approximation and (\ref{linear rho}).
By Stirling's approximation $((1+\del)k)!\leq C(k!)^{1+\del}(1+\del)^{(1+\del)k}$
and (\ref{GorEst2}) follows now by \ref{GorEst1} and the inequality $(1+\del)^{\frac1{1+\del}}\leq e$. 
\end{proof}

%\begin{remark}
%Something about graph with certain type of "strong local dependence"?
%\end{remark}

\subsection{Proof the Theorem \ref{NoncCum}}

Fix some $N\geq1$ and set $V=V_N=\{1,2,...,N\}$. For any $n,m\in V$ set 
\[
\rho(n,m)=\rho_{\ell}(n,m)=\min_{1\leq i,j\leq \ell}|in-jm|.
\] 
Then for any $\Del_1,\Del_2\subset V$,
\begin{equation}\label{dis-rel}
\rho(\Del_1,\Del_2)=\inf\{|x-y|:\, x\in\cT_1,\,y\in\cT_2\}:=\text{dist}(\cT_1,\cT_2)
\end{equation}
where $\cT_i=\{jt:\,t\in\Del_i,\,1\leq j\leq\ell\},\,i=1,2$.
Moreover, for any $s\geq 1$ and $1\leq n\leq N$, 
\begin{equation*}
A_s(n,N):=\{m\in V:\,\rho(m,n)\leq s\}=\bigcup_{1\leq i,j\leq\ell}\left[\frac{in-s}j,\frac{in+s}j\right]
\end{equation*}
and so 
\begin{equation}\label{u0 cond ver}
|A_s(n,N)|\leq 3\ell^2 s.
\end{equation}
Therefore (\ref{linear rho}) holds true in our situation with $u_0=1$. 
For each $n\in V$ put $\Theta_n=(\xi_n,\xi_{2n},...,\xi_{\ell n})$ 
and 
\[
X_n=F(\Theta_n)-\bbE F(\Theta_n).
\]
Then $\bar S_N=\sum_{n\in V}X_n$.
We will verify that the remaining assumptions of Corollary \ref{GorcCor} hold true 
with the above $X_n$'s.
First, for each $r\geq0$ and $n\geq1$, set 
$\xi_{n,r}=\bbE[\xi_n|\cF_{n-r,n+r}]$, 
$\Theta_{n,r}=(\xi_{n,r},\xi_{2n,r},...,\xi_{\ell n,r})$ and 
\[
X_{n,r}=F(\Theta_{n,r})-\bbE F(\Theta_{n,r}).
\]
Set $\rho_\infty=2K(1+\ell)$ and  $\varrho_{t}=2K(1+\ell\tau_{\lambda t}^\lambda),\,0<t<\infty$. 
When $\lambda=0$ then by (\ref{F Bound}) for any $n\geq1$ and $r\geq0$,
\begin{equation}\label{var sigs app infty}
\max(\|X_n\|_\infty,\|X_{n,r}\|_\infty)\leq 2K(1+\ell)=\varrho_{\infty}
\end{equation}
while when $\lambda>0$ we derive similarly that for any $0<t<\infty$, $n\geq 1$ and $r\geq0$,
\begin{equation}\label{var sigs app}
\max(\|X_n\|_t,\|X_{n,r}\|_t)\leq 2K(1+\ell\tau_{\lambda t}^\lambda)=\varrho_{t}
\end{equation}
where we also used the contraction of conditional expectations.  
Note that $\varrho_{t_1}\leq \varrho_{t_2}$ whenever $0<t_1\leq t_2<\infty$.
In our future applications of Corollary \ref{GorcCor} we will always take $\varrho_{v,\infty}=\varrho_{\infty}$ and  $\varrho_{v,t}=\varrho_{t}$ for
$0<t<\infty$.

Next, 
when (\ref{Moment growth assumptions}) holds true and $\lambda>0$
then by Stirling's approximation there exists an 
absolute constant $C>1$ so that for any $k\geq1$,
\begin{equation}\label{Moms app!}
\tau_{\lambda k}^{\lambda k}=\bbE|\xi_1|^{k\lambda}\leq M^{k\lambda}((k\lambda)!)^{\zeta}\leq 
C^{\zeta(\lambda+1)}Q^k(k!)^{\lambda\zeta}\leq(C^{\zeta(\lambda+1)}Q)^k(k!)^{\lambda\zeta}
\end{equation}
where $Q=\lambda^{\zeta\lambda}M^\lambda\geq1$. Therefore, the collection of numbers
$\varrho_{v,k}=\varrho_{k}$ 
satisfies (\ref{MomGorch}) with $4K\ell C^{\zeta(\lambda+1)}Q$ in place of $M$ and with $\te=\lambda\zeta$. 
%When $\lambda=0$ they satisfy this condition with $2(K+\ell)$ and $\te=0$.

Now we will verify condition (\ref{MixGorc}). We will need first the following
general result. Let $U_i,\,i=1,2,...,L$ be  $d_i$-dimensional random vectors defined on the
probability space $(\Om,\cF,P)$ from Section \ref{sec1}, and
$\{\cC_j: 1\leq j\leq s\}$ be a partition of $\{1,2,...,L\}$. 
Consider the random vectors $U(\cC_j)=\{U_i: i\in\cC_j\}$, $j=1,...,s$, 
and let 
\[
U^{(j)}(\cC_i)=\{U_i^{(j)}: i\in\cC_j\},\,\, j=1,...,s
\]
be independent copies of the $U(\cC_j)$'s.
For each  $1\leq i\leq L$ let $a_i\in\{1,...,s\}$ be the unique index
such that $i\in\cC_{a_i}$, and 
for any  bounded Borel function $H:\bbR^{d_1+d_2+...+d_L}\to\bbR$ set
\begin{equation}\label{ExpDiffDef}
\cD(H)=\big|\bbE H(U_1,U_2,...,U_L)-\bbE H(U_1^{(a_1)},U_2^{(a_2)},...,U_L^{(a_L)})\big|.
\end{equation}
The following result is proved in Corollary 1.3.11 in \cite{book}
(see also Corollary 3.3 in \cite{Ha}),

\begin{lemma}\label{lem3.1-StPaper}
Suppose that each $U_i$ is $\cF_{m_i,n_i}$-measurable, where $n_{i-1}<m_i\leq n_i<m_{i+1}$, 
$i=1,...,L$, $n_0=-\infty$ and $m_{L+1}=\infty$.  
Then, for any bounded
Borel function $H:\bbR^{d_1+d_2+...+d_L}\to\bbR$,
\begin{equation}\label{Lem 3.1 res}
\cD(H)\leq 4\sup|H|\sum_{i=2}^L\phi(m_i-n_{i-1})
\end{equation}
where $\sup|H|$ is the supremum of $|H|$. In particular, when $s=2$ then
\begin{equation}\label{alpha cor}
\al\big(\sig\{U(\cC_1)\},\sig\{U(\cC_2)\}\big)\leq4\sum_{i=2}^L\phi(m_i-n_{i-1})
\end{equation}
where $\sig\{X\}$ stands for the $\sig$-algebra generated by a random variable 
$X$.
\end{lemma}

Next, in order to show that (\ref{MixGorc}) holds true we first notice that
 for any set of pairs $(a_i,b_i),\,i=1,2,...,m$,
\begin{equation}\label{prod-sum relation}
\prod_{i=1}^ma_i-\prod_{i=1}^m b_i=\sum_{i=1}^m\prod_{1\leq j<i}a_j(a_i-b_i)\prod_{i<j\leq m}b_j.
\end{equation}
Let $n_1,...,n_m\in V$ and $q\geq0$.
When $\lambda=0$ using (\ref{prod-sum relation}), (\ref{var sigs app infty}) and (\ref{F Hold}) we obtain that 
for each $1\leq i\leq m$,
\begin{eqnarray}\label{Approx1}
\left|\bbE\prod_{i=1}^mX_{n_i}-\bbE\prod_{i=1}^mX_{n_i,q}\right|\leq\\
m(\varrho_{\infty})^{m-1}\max\{\bbE |X_{n_i}-X_{n_i,q}|,\,1\leq i\leq m\}\leq
m(\varrho_{\infty})^{m}\ell\be_\ka^{\ka}(q).\nonumber
\end{eqnarray}
When $\lambda>0$ then by  the contraction of conditional expectations for any $1\leq i\leq m$,
\begin{eqnarray*}
\|X_{n_i}-X_{n_i,q}\|_m\leq K\big\|1+\sum_{j=1}^\ell(|\xi_{jn_i}|^\lambda+|\xi_{jn_i,q}|^\lambda)\big\|_m
\sum_{j=1}^\ell\big\||\xi_{jn_i}-\xi_{jn_i,q}|^\ka\big\|_\infty\\
\leq K\ell(1+2\ell\tau_{m\lambda}^\lambda)\be_\infty^\ka(q)\leq \varrho_m\ell\be_\infty^\ka(q). 
\end{eqnarray*}
where $\varrho_{m}$ is defined in \ref{var sigs app}.
Therefore by  (\ref{prod-sum relation}), (\ref{var sigs app}), (\ref{F Hold}) and the H\"older inequality,
\begin{equation}\label{Approx2}
\left|\bbE\prod_{i=1}^mX_{n_i}-\bbE\prod_{i=1}^mX_{n_i,q}\right|\leq
m(\varrho_{m})^m\ell\be_\infty^\ka(q).
\end{equation}

Now, let $k,b\geq 1$ and a finite collection $A_j,\, j\in\cJ$ of nonempty subsets of $V$ be so that
$r:=\sum_{j\in\cJ}|A_j|\leq k$ and $\rho(A_j,A_i)\geq b$ whenever $i\not=j$. 
Set $q_b=[\frac b3]$.  
When $\lambda=0$  set $\del=\infty$ and
\[
\gam(b,r)=\gam_\infty(b,r)=128\ell r\big(\phi(q_b)+\be_\ka^\ka(q_b)\big)
\]
while when $\lambda>0$  set $\del=1$ and 
\[
\gam(b,r)=\gam_{1}(b,r)=128\ell r\big(\phi^{\frac12}(q_b)+\be_\infty^\ka(q_b)\big).
\] 
We claim that in both cases (\ref{MixGorc}) holds true with $\varrho_{v,t}=\varrho_{t}$ defined 
in (\ref{var sigs app}) and (\ref{var sigs app infty}) and the above $\del$ and $\gam_\del(b,r)$
(depending on the case). Indeed, when $\lambda=0$ and $\del=\infty$  set
$\gam'_\del(b,r)=32\ell r\phi(q_b)$, while when $\lambda>0$ and $\del=1$ we set 
$\gam'_\del(b,r)=32\ell r(\phi(q_b))^{\frac\del{1+\del}}=32\ell r\sqrt{\phi(q_b)}$. In order to prove this claim 
we first assert that in both cases,
\begin{equation}\label{Suf Show}
\left|\bbE\prod_{j\in\cJ}\prod_{i\in A_j}X_{i,q_b}-\prod_{j\in\cJ} \bbE\prod_{j\in A_j}X_{i,q_b}\right|\leq
(r-1)\Big(\prod_{j\in\cJ}\prod_{i\in A_j}\|X_{i,q_b}\|_{(1+\del)k}\Big)\gam'_\del(b,k).
\end{equation}
It is clear that (\ref{MixGorc}) with these $A_j$'s, $b$ and $k$ follow from
either (\ref{Approx1}) and (\ref{Suf Show}) or (\ref{Approx2}) and (\ref{Suf Show}), depending on the
case, where when $r\geq 2$ we use that $r\leq2(r-1)$.
In order to obtain (\ref{Suf Show}) 
we need first the following. 
Let $\Del_1,\Del_2\subset\bbN$ be so that $\rho(\Del_1,\Del_2)\geq b$
and set $d_1=|\Del_1|+|\Del_2|$ and $\cT_i=\{jx:\,x\in \Del_i, 1\leq j\leq\ell\}$, $i=1,2$. 
Then by (\ref{dis-rel}) we have $\text{dist}(\cT_1,\cT_2)=\rho(\Del_1,\Del_2)\geq b$ and so
we can write
\[
\cT:=\cT_1\cup \cT_2=\bigcup_{i=1}^{L}C_i
\]  
where $L\leq \ell d_1$, $c_i+b\leq c_{i+1}$ for any $c_i\in C_i$ and $c_{i+1}\in C_{i+1}$,\, $i=1,2,...,L-1$
and each one of the $C_i$'s is either a subset of $\cT_1$ or a subset of $\cT_2$. Applying 
(\ref{alpha cor}) with the random vectors $U_i=\{\xi_{j,q_b}:\,j\in C_i\},\,i=1,2,...,L$ and the 
partition of $\{1,2,...,L\}$ into the sets $\cC_1=\{1\leq i\leq L:\,C_i\subset \cT_1\}$ and 
$\cC_2=\{1\leq i\leq L:\,C_i\subset \cT_2\}$  we obtain that 
\begin{equation}\label{alpha cor'}
\al\big(\sig\{X_{i,q_b}: i\in \Del_1\}, \sig\{X_{j,q_b}: j\in \Del_2\}\big)\leq 4\ell d_1\phi(q_b).
\end{equation}
Recall next that (see Corollary A.2 in \cite{HallHyde}) for any two sub-$\sig$-algebras $\cG,\cH\subset\cF$,
\begin{equation}\label{alpha cov HallHyde}
\text{Cov}(\eta_1,\eta_2)\leq 8\|\eta_1\|_u\|\eta_2\|_v\big(\al(\cG,\cH)\big)^{1-\frac1u-\frac1v},
\end{equation}
whenever $h_1$ is  $\cG$-measurable, $h_2$ is $\cH$-measurable
and $1<u,v\leq\infty$ satisfy that $\frac1u+\frac1v<1$ (where we set $\frac1\infty=0$). 
The estimate (\ref{Suf Show}) follows now exactly as in the paragraph preceeding equality (10) in \cite{Gorc}, 
relying on (\ref{alpha cov HallHyde}) and on (\ref{alpha cor'}), in place of the mixing conditions
from \cite{Gorc}. Indeed, writing $\cJ=\{1,2,...,J\}$, setting $\Del_1=A_1$ and 
$\Del_2=\bigcup_{1<i\leq J}A_i$ and applying (\ref{alpha cov HallHyde}) with 
$u=\frac{(1+\del)k}{|\Del_1|}$ and $v=\frac{(1+\del)k}{|\Del_2|}$
we obtain that 
\begin{equation}\label{Temp prod  est}
\Big|\text{Cov}\big(\prod_{i\in \Del_1}X_{i,q_b},\prod_{i\in\Del_2}X_{i,q_b}\big)\Big|\leq
8\big\|\prod_{i\in \Del_1}X_{i,q_b}\big\|_u\,\big\|\prod_{i\in\Del_2}X_{i,q_b}\big\|_v\al^{1-\frac1{1+\del}}
\end{equation}
where $\al=\al\big(\sig\{X_{i,q_b}: i\in \Del_1\}, \sig\{X_{j,q_b}: j\in \Del_2\}\big)$
and we also used that $\al\leq1$ and
\[
\frac1u+\frac 1v=\frac{|\Del_1\cup\Del_2|}{k(1+\del)}=\frac{r}{k(1+\del)}\leq\frac 1{1+\del}.
\]
Using the H\"older inequality to estimate the norms on the right hand side of (\ref{Temp prod  est}) 
and then repeating the above arguments with $\cJ_i=\{i,i+1,...,J\},\,i=2,3,...,J$ in place of $\cJ$ we
obtain (\ref{Suf Show}), taking into account that $J=|\cJ|\leq \sum_{i\in\cJ}|A_i|=r$.
Using either (\ref{var sigs app infty}) or (\ref{var sigs app}) we conclude that 
all the conditions of Corollary \ref{GorcCor} are satisfied under either Assumption \ref{ass1} or 
Assumption \ref{ass2}, and the proof of Theorem \ref{NoncCum} is complete.

\subsection{Proof of Theorems \ref{Thm:ModDevNonc} and \ref{Thm2.4}}.
First, (\ref{FirstExpCon}) from Theorem \ref{Thm:ModDevNonc} follows from Theorem \ref{NoncCum} and
Lemma 2.3 in \cite{SaulStat}. Next, for the purpose of proving Theorem \ref{Thm2.4}, suppose that $D^2>0$. Then
(\ref{ModDev3}) follows by Lemma 6.2 in \cite{Dor} (which is a consequence of Lemma 2.3 in \cite{SaulStat}).
Finally, let $a_N,\,N\geq 1$ be a sequence of real numbers so that 
\[
\lim_{N\to\infty}a_N=\infty\,\,\text{ and }\,\,\lim_{N\to\infty}{a_N}{N^{-\frac{1}{2+4\gam}}}=0
\]
where $\gam=\gam_1=\frac1\eta$ under Assumption \ref{ass1} and 
$\gam=\gam_2=\gam_1+\lambda\zeta$ under Assumption \ref{ass2}.
The variances $v_N$ grow linearly fast in $N$ and therefore
by Theorem \ref{NoncCum} and Theorem 1.1 in \cite{Dor} the sequence $(a_N)^{-1}Z_N,\,N\geq 1$ satisfies the
MDP with the speed $s_N=a_N^2$ and the rate function 
$I(x)=\frac12 x^2$. Since $v_N/N$ converges to 
$D^2>0$ as $N\to\infty$, $|\bbE S_N|$ is bounded in $N$ and $I$ is continuous we derive 
that $(DN^{\frac12}a_N)^{-1}S_N,\,N\geq1$
satisfies the MDP stated in Theorem \ref{Thm2.4}, and the
proof of Theorem \ref{Thm2.4} is complete.

\subsection{Product functions case}\label{ProdCase}
Consider the situation when $F$ has the form 
\[
F(x_1,...,x_\ell)=\prod_{i=1}^\ell f_i(x_i).
\]
We will describe here shortly how to prove Theorems \ref{Thm:ModDevNonc} and \ref{Thm2.4} in the 
situations discussed at the end of Section \ref{sec2}.
\subsubsection{$\al$-mixing case}\label{AlMixCase}
First, in the notations of Lemma  \ref{lem3.1-StPaper}, we obtain that (\ref{Lem 3.1 res}) holds true for functions of the form
$H(u)=\prod_{i=1}^{L}g_i(u_i)$ when all of the $g_i$'s are bounded, where $\phi(m_i-n_{i-1})$
is replaced by $4\al(m_i-n_{i-1})$ for $i=2,3,...,L$. 
Indeed, setting 
\[
u^{(\cC_j)}=\{u_i: i\in\cC_j\}\,\,\text{ and }\,\,G_j(u^{(\cC_j)})=\prod_{i\in\cC_j}g_i(u_i),\,j=1,2,...,s,
\]
we derive from (\ref{Aapha block result0}),
exactly as in the proof of Corollary 1.3.11 in \cite{book} (or Corollary 3.3 in \cite{Ha}),  that
\begin{equation}\label{Aapha block result}
\Big|\bbE H(U_1,...,U_L)-\prod_{j=1}^s\bbE G_j(U(\cC_j))\Big|\leq 16\big(\prod_{j=1}^L\sup|g_j|\big)
\sum_{i=2}^{L}\al(m_i-n_{i-1}).
\end{equation}
 Note that the derivation of (\ref{Aapha block result}) is indeed possible 
since (\ref{Aapha block result0}) holds true for arbitrary
bounded $g_i$'s,  appropriate $U_i$'s and partitions $\cC$'s.
Relying on (\ref{Aapha block result}) 
we can approximate the left-hand side of (\ref{MixGorc})
and therefore the results  stated in Theorems \ref{Thm:ModDevNonc} and \ref{Thm2.4} hold true with $\alpha(n)$ in place of $\phi(n)$.

We remark that (\ref{Aapha block result0}) follows, in fact, by 
a repetitive application of (\ref{alpha cov HallHyde}) with $u=v=\infty$.
Applying (\ref{alpha cov HallHyde}) with finite $u$'s and $v$'s we obtain similar
estimates when the $g_i(U_i)$'s are not bounded but only satisfy certain moment
conditions, where the product $\prod_{j=1}^L\sup|g_j|$ is replaced with an
appropriate product of the form $\prod_{i=1}^L\|g_i(U_i)\|_q$ and $\al(m_t-n_{t-1})$ is replaced
with $(\al(m_t-n_{t-1}))^{\zeta}$ for an appropriate $0<\zeta<1$. We refer the readers to the proof of
(\ref{Suf Show}) for the exact details.
Relying on this ``unbounded version" of (\ref{Aapha block result}),
we can approximate the left-hand side of (\ref{MixGorc})
and obtain results similar to the ones  stated in Theorem \ref{Thm:ModDevNonc} (ii) and Theorem \ref{Thm2.4} (ii), but with with $\alpha(n)$ in place of $\phi(n)$.

\subsubsection{Decay of correlations case}\label{DeCor}
Let $T,\cH$ and $c(m),m\geq 1$ be as described at
the end of Section \ref{sec2}. 
Let  $n_1<n_2<...<n_L$ and $g_1,...,g_L\in\cH$.
In the notations of Lemma  \ref{lem3.1-StPaper}, 
using (\ref{DecCor block result0}) and the $T$-invariance of $P$, we obtain  similarly to the above $\alpha$-mixing case that
\begin{equation}\label{DecCor block result}
\left|\bbE_P\prod_{i=1}^L g_i\circ T^{n_i}-\prod_{j=1}^s\bbE_P\prod_{i\in\cC_j}g_i\circ T^{n_i}
\right|\leq 2dM^L\sum_{t=2}^L c(n_t-n_{t-1})
\end{equation}
where $M=\max\{\sup|g_i|,\|g_i\|_\cH:\,i=1,2,...,L\}$. 
Note that 
when $\sum_{n=1}^\infty nc(n)<\infty$ then 
all the results stated in Theorem \ref{D-thm}
are proved similarly to \cite{KV}, \cite{HK1} and \cite{Ki3} relying on
(\ref{DecCor block result}) instead of the mixing assumptions from there. The inequality
(\ref{DecCor block result}) also yields appropriate
estimates of the left-hand side of (\ref{MixGorc}), and we conclude that
that all the results  stated in Theorem \ref{Thm:ModDevNonc} (i)  and Theorem \ref{Thm2.4} (i)
hold true with  $\beta_\ka(n)\equiv0$ and $c(n)$ in place of $\phi(n)$.

\section{Exponential inequalities via martingale approximation-proof of Theorems \ref{MartExp-Cor} and \ref{MartExp}}\label{sec4}\setcounter{equation}{0}
In this section we adapt the martingale approximation technique from \cite{KV}
and approximate $S_N$ in the $L^\infty$ norm by martingales with bounded differences.
As in \cite{KV} we first write
\begin{equation}\label{F rep}
F(x_1,...,x_\ell)=\sum_{i=1}^\ell F_i(x_1,...,x_i)
\end{equation}
where 
\[
F_\ell(x_1,...,x_\ell)=F(x_1,...,x_\ell)-\int F(x_1,...,x_{\ell-1},z)d\mu(z)
\]
and for $i=1,2,...,\ell-1$,
\begin{eqnarray*}
F_i(x_1,...,x_i)=\int F(x_1,...,x_i,z_{i+1},...,z_\ell)d\mu(z_{i+1})...d\mu(z_\ell)-\\
\int F(x_1,...,x_{i-1},z_{i},...,z_\ell)d\mu(z_{i})...d\mu(z_\ell).
\end{eqnarray*}
Then for each $1\leq i\leq\ell$,
\[
\int F_i(y_1,...,y_{i-1},z)d\mu(z)=0,\,\,\forall\,y_1,...,y_{i-1}
\]
where for $i=1$ we used that $\bar F=0$.

Next,  recall  that  (see \cite{Br}, Ch. 4) for any two sub-$\sigma$-algebras $\cG,\cH\subset\cF$,
\begin{equation}\label{Phi-Rel-StPaper}
2\phi(\cG,\cH)=\sup\{\|\bbE[g|\cG]-\bbE g\|_{\infty}\,:  g\in L^\infty(\Om,\cH,P),\,
\|g\|_{\infty}\leq1\}
\end{equation}
where $\phi(\cG,\cH)$ is defined by (\ref{Phi Gen}). 
The following result is a version of Corollary 3.6 in \cite{KV} 
and Lemma 1.3.10 in \cite{book} (see also Lemma 3.2 in \cite{Ha}).
It does not seem to be new but for readers' convenience and completeness
we will prove it here.

\begin{lemma}\label{thm3.11-StPaper}
Let $\cG,\cH\subset\cF$ be two sub-$\sigma$-algebras of $\cF$ and $d\in\bbN$. Let 
$f(\cdot,\om):\bbR^d\to\bbR$ be a random function so that $f(x,\om)$ is $\cH$-measurable for
any fixed $x\in\bbR^d$ and $P$-a.s. for any $x,y\in\bbR^d$,
\begin{equation}\label{rand hold}
|f(x,\om)|\leq C\,\,\,\text{ and }\,\,\,|f(x,\om)-f(y,\om)|\leq C|x-y|^\ka
\end{equation}
where $C>0$ and $\ka\in(0,1]$ are constants which do not depend on $x,y$ and $\om$.
Set $\tilde f(x,\om)=\bbE[f(x,\cdot)|\cG](\om)$ and
$\bar f(x)=\int f(x,\om)dP(\om)=\int \tilde f(x,\om)dP(\om)$. Then there exists a
measurable set $\Om'\subset\Om$ so that $P(\Om')=1$, $\,\tilde f(x,\om)$ is defined 
for all $\om\in\Om'$ and $x\in\bbR^d$ and
\begin{equation}\label{Meas-StPaper.0}
\sup_{x\in\bbR^d}|\tilde f(x,\om)-\bar f(x)|
\leq2C\phi(\cG,\cH),\,\,P-a.s.
\end{equation}
In particular, for any $\bbR^d$-valued random 
variable $X$,
\begin{equation}\label{Meas-StPaper}
|\tilde f(X,\om)-\bar f(X)|
\leq2C\phi(\cG,\cH),\,\,P-a.s.
\end{equation}
\end{lemma}

\begin{proof}
Let $\cA=\{\cA_i: i\in \cI\}$ be a countable partition of $\bbR^d$ and denote its diameter by 
$\text{diam}\cA$. For each $i\in\cI$ let 
$\mathds{1}_{A_i}$ be the indicator function of $A_i$ and
choose some $a_i\in A_i$. Then by (\ref{rand hold}), $P$-a.s. for any $x\in\bbR^d$ we have
\[
|f(x,\om)-\sum_{i\in\cI}\mathds{1}_{A_i}(x)f(a_i,\om)|\leq C(\text{diam}\cA)^\ka.
\]
Taking conditional expectations with respect to $\cG$ and then the limit as $\text{diam}\cA\to 0$ we
obtain the existence of $\Om'$ as in the statement of the 
lemma. Fixing $\cA$ and taking again conditional expectations with respect to $\cG$ we derive that
\begin{eqnarray*}
\sup_{x\in\bbR^d}|\tilde f(x,\om)-\sum_{i\in\cI}\mathds{1}_{A_i}(x)\tilde f(a_i,\om)|\leq C(\text{diam}\cA)^\ka,\,\,P-\text{a.s.}.
\end{eqnarray*}
Similarly, we obtain by taking expectations that
\[
\sup_{x\in\bbR^d}|\bar f(x)-\sum_{i\in\cI}\mathds{1}_{A_i}(x)\bar f(a_i)|\leq C(\text{diam}\cA)^\ka.
\]
Using (\ref{Phi-Rel-StPaper}) and (\ref{rand hold}) we deduce
that for each $i$,
\[
|\tilde f(a_i,\om)-\bar f(a_i)|\leq 2\|f(a_i,\cdot)\|_{\infty}\phi(\cG,\cH)
\leq 2C\phi(\cG,\cH),\,\,P-\text{a.s.}
\]
and therefore, $P$-a.s.,
\[
\sup_{x\in\bbR^d}|\tilde f(x,\om)-\bar f(x)|\leq  2C\phi(\cG,\cH)+2C(\text{diam}\cA)^\ka.
\]
Taking the limit as $\text{diam}\cA\to 0$ we obtain (\ref{Meas-StPaper.0}).
\end{proof}

Next, consider the random functions $F_{i,n,r}$ given by
\[
F_{i,n,r}(x_1,...,x_{i-1},\om)=\bbE[F_i(x_1,...,x_{i-1},\xi_n)|\cF_{n-r,n+r}](\om).
\]
Note that in view of the uniform continuity of $F$  these are indeed random
functions, i.e. all the random variables 
$F_{i,n,r}(x_1,...,x_{i-1},\cdot),\,x_1,...,x_{i-1}\in\bbR^{\wp}$ can be 
defined on a measurable set $\Om'$ so that $P(\Om')=1$.  
Set
\begin{eqnarray}
Y_{i,in}=F(\xi_n,\xi_{2n},...,\xi_{in})\,\text{ and }\,Y_{i,m}=0\text{ if }m\not\in\{in: n\in\bbN\}
\,\,\text{ and}\\
Y_{i,in,r}=F_{i,in,r}(\xi_{n,r},\xi_{2n,r},...,\xi_{(i-1)n,r},\om)\,\text{ and }\,Y_{i,m,r}=0
\text{ if }m\not\in\{in: n\in\bbN\}\nonumber
\end{eqnarray} 
where we recall that $\xi_{m,r}=\bbE[\xi_m|\cF_{m-r,m+r}]$ for any $m\geq1$.

The  following result is proved exactly as in the proof of Proposition 5.8 in \cite{KV}
using Lemma \ref{thm3.11-StPaper} and the inequality $|F|\leq K(1+\ell)$ 
instead of Corollary 3.6 (ii) and the moment assumptions from there.
\begin{corollary}\label{MartConstCor}
Suppose that $\varphi:=\sum_{n=0}^\infty\phi(n)<\infty$. Then there exists a constant $B>0$ which depends only 
on $\ell$ so that for any $l\geq0$ and $r\geq 0$,
\[
\sum_{n=l}^\infty\|\bbE[Y_{i,n,r}|\cF_{-\infty,l+r}]\|_\infty\leq BK(r+1+\varphi).
\]
\end{corollary}
Now we introduced the martingales constructed in \cite{HK1} relying 
on ideas originated in \cite{KV}.  For any $1\leq i\leq\ell$, $n\geq0$ and $r\geq0$ set 
$R_{i,n,r}=\sum_{s\geq n+1}\bbE[Y_{i,s,r}|\cF_{-\infty,n+r}]$ and 
\[
W_{i,n,r}=Y_{i,n,r}+R_{i,n,r}-R_{i,n-1,r}.
\]
Then when $i$ and $r$ are fixed $W_{i,n,r},\, n\geq 1$ is a martingale difference with 
respect to the filtration $\{\cF_{-\infty,n+r}:\,n\geq1\}$ and by
Corollary \ref{MartConstCor},
\begin{equation}\label{R bound}
\|R_{i,n,r}\|_\infty\leq 2BK(\varphi+r+1)
\end{equation}
and therefore there exists a constant $B_1>0$ which depends only on $\ell$ so that
\[
\|W_{i,n,r}\|_\infty\leq B_1K(\varphi+r+1).
\]
Set 
$
W_{i,n,r}^{(N)}=\mathds{1}_{\{n\leq iN\}}W_{i,n,r},
$
\[
W_{n,r}^{(N)}=\sum_{i=1}^\ell W_{i,n,r}^{(N)},
\] 
 $M_{i,n,r}^{(N)}=\sum_{m=1}^nW_{i,m,r}^{(N)}$
and 
\[
M_{n}^{(N,r)}=\sum_{m=1}^nW_{m,r}^{(N)}
=\sum_{i=1}^{\ell}M_{i,n,r}^{(N)}.
\]
Then when $r$ and $N$ are fixed $M_{n}^{(N,r)},\,n\geq 1$ is a martingale 
(with respect to the above filtration)
whose differences are bounded by $\ell B_1K(\varphi+r+1)$. We estimate now 
the $L^\infty$-norm 
\[
\|S_N-M_{N\ell}^{(N,r)}\|_{\infty}.
\]
We first write
\[
S_N-M_{N\ell}^{(N,r)}=\sum_{i=1}^{\ell}
\sum_{n=1}^{N}(Y_{i,in}-Y_{i,in,r})+\sum_{i=1}^\ell(R_{i,N\ell,r}-R_{i,0,r})
\]
where we used (\ref{F rep}).
By replacing $\xi_{jn}$ with $\xi_{jn,r}$, $j=1,2,...,i$ in the definitions of 
$Y_{i,in}$ and $Y_{i,in,r}$, using the H\"older continuity of $F$ and that 
$\xi_{\ell n,r}$ is $\cF_{\ell n-r,\ell n+r}$-measurable we obtain that
\[
|Y_{i,in}-Y_{i,in,r}|\leq KB_2\be_\infty^\ka(r),\,P-\text{a.s.}
\]
for any $1\leq i\leq \ell$, $n\in\bbN$ and $r\geq0$, where $B_2=B_2(\ell)$ is some constant 
which depends only on $\ell$. Combining this with (\ref{R bound}) we obtain that
\begin{equation}\label{del 2'}
\|S_N-M_{N\ell}^{(N,r)}\|_{\infty}\leq 
B_3K(N\be_\infty^\ka(r)+\varphi+r+1\big):=\del_2'
\end{equation}
where $B_3=B_3(\ell)$ is another constant, and the proof of Theorem \ref{MartExp} is complete. 
In order to prove Theorem \ref{MartExp-Cor},  we first
apply the Hoeffding-Azuma inequality (see, for instance, page 33 in \cite{Mil}) and obtain 
that for any $\la>0$, 
\[
\bbE e^{\la M_{N\ell}^{(N,r)}}\leq e^{\la^2\sum_{n=1}^{\ell N}\|W_{n}^{(N,r)}\|_\infty^2}
\leq e^{\ell N\del_0^2\la^2}
\] 
where $\del_0= B_1K(\varphi+r+1)$. Combining this with (\ref{del 2'}) we obtain (\ref{ExpEstSN}).
Next, by the Markov inequality
 for any random variable $Z$, $t_0>0$ and $\la>0$ we have
$P(Z\geq t_0)\leq e^{-\la t_0}\bbE e^{\la Z}$. Taking $Z=S_N$, $t_0=t+\del_2$, using (\ref{ExpEstSN})
and then optimizing by taking $\la=\frac{t}{2\ell N\del_2^2}$ we obtain (\ref{ExpConc}), and the 
proof of Theorem \ref{MartExp-Cor} is complete.

%\section{Extensions and special cases}\label{sec5}\setcounter{equation}{0}
%\subsection{local dependence case}
%In the case when for any $n$ there exists a set $\cA_n$ of indexes 
%so that $\xi_n$ is independent of $\xi_m, m\not\in A_n$ and $A_n$ is 
%bounded in $n$ (for instance, $m$-dependence)...$\gam=1$...best possible
%(need to talk about the variance)

\section{Nonlinear indexes}\label{NonLine}
\label{sec5}\setcounter{equation}{0}
Let $q_1,...,q_\ell$ be functions which map $\bbN$
to $\bbN$, are strictly increasing on some ray $[R,\infty)$ and are ordered so that
\[
q_1(n)<q_2(n)<...<q_\ell(n)
\]
for any sufficiently large $n$.
For any $N\in\bbN$  consider the random variable
\begin{equation}\label{Z N nonlinear indexes}
S_N=\sum_{n=1}^{N}\big(F(\xi_{q_1(n)},\xi_{q_2(n)},...,
\xi_{q_\ell(n)})-\bar F\big)
\end{equation}
where $\bar F$ is given by (\ref{F bar}).
We further assume that the difference
$q_i(n)-q_{i-1}(n)$ tends to $\infty$ as $n\to\infty$ for any $i=1,2,...,\ell$, where $q_0\equiv 0$,
though  the situation when some of these differences are nonnegative constants can be considered,
as well (see Section 3 in \cite{HK2}). 
Next, for any $n,m\in\bbN$ set
\[
\tilde\rho(n,m)=\tilde\rho_\ell(n,m)=\min_{1\leq i,j\leq\ell}|q_i(n)-q_j(m)|.
\]
We will rely on the following  
\begin{assumption}\label{Ass 5.1}
There exists $Q\geq1$ so that for any $1\leq j\leq\ell $ and $a,b\geq q_j(R)$,
\begin{equation}\label{rest1}
|q_j^{-1}(a)-q_j^{-1}(b)|\leq Q(1+|a-b|)
\end{equation}
where $q_j^{-1}$ is the inverse of the restriction of $q_j$ to the ray 
$[R,\infty)$.
\end{assumption}\
Set
$\tilde A_s(n,N)=\{1\leq m\leq N: \tilde \rho(n,m)\leq s\}$. 
When (\ref{rest1}) holds true then for any $1\leq n\leq N$ and $s\geq1$,
\[
|\tilde A_s(n,N)|\leq Q\ell^2(1+2s)\leq3\ell^2Qs
\]
which means that (\ref{linear rho}) holds true with $c_0=3Q\ell^2$ and $u_0=1$.
Condition (\ref{rest1}) holds true, for instance, when all $q_j$'s have the form $q_j(x)=[p_j(x)]$ 
where each $p_j$ is a strictly increasing function whose inverse $p_j^{-1}$ 
has bounded derivative on some ray $[K,\infty)$. For example we
can take $p_j$'s to be a polynomial with positive leading
coefficient, exponential function etc.

We conclude that under Assumption \ref{Ass 5.1}, all  the results stated
in Theorem \ref{NoncCum} hold true. Therefore, (\ref{FirstExpCon}) holds true
and all the results stated in 
Theorem \ref{Thm2.4}  hold true when $D^2$ exists and it is positive. 
The limit $D^2$ exists when $q_i$'s satisfy 
the conditions from \cite{KV} or, as in \cite{HK2}, when they are polynomials
taking integer values on the integers. See \cite{HK1} and \cite{HK2} for conditions
equivalent to $D^2>0$. Note also that for such $q_i$'s Theorem \ref{MartExp}
holds true, as well, since the martingale approximation method was applied in \cite{KV} and 
\cite{HK1} successfully, and so the arguments from Section \ref{sec4} can be repeated.

\begin{remark}\label{NonLineRem}
Let $q(n),\,n\geq 1$ be a strictly increasing sequence of natural numbers, and 
consider the process $\tilde\xi_n,\,n\geq 1$ given by $\tilde\xi_n=\xi_{q(n)}$.
Set $\tilde\cF_{m,n}=\cF_{q(m),q(n)}$ and let $\tilde\phi(n)$ and $\tilde\be_q(n)$ be defined similarly
to $\phi(n)$ and $\be_q(n)$ but with  the $\tilde\cF_{m,n}$'s in place of the $\cF_{m,n}$'s. 
Then $\tilde\be_q(n)\leq \be_q(q(n))$ and $\tilde\phi(n)\leq\phi(j(n))$, where 
\[
j(n)=\inf_{m\geq1}(q(m+n)-q(m)).
\]
When $q(n),j(n)\geq cn^l$ for some $l\geq 2$ and $c>0$ then the
mixing and approximation coefficients $\tilde\phi(n)$ and $\tilde\be_q(n)$ converge to $0$ 
faster than $\phi(n)$ and $\be_q(n)$, and by writing $s=q(s')\geq c(s')^l$ 
we can take $u_0=\frac 1l$ in (\ref{linear rho}).
Repeating the arguments from the proof of 
Theorem \ref{NoncCum}, we obtain similar estimates of $|\Gam_k(\bar S_N)|$, but with 
$\gam_1'=\frac{1}{\eta l^2}<\gam_1$ in place of $\gam_1=\frac1\eta$. 
The assumption that the distribution of $(\xi_n,\xi_m)$ depends only
on $n-m$ was only needed in order for $D^2$ to exist and for obtaining convergence rate towards it.
Therefore, (\ref{FirstExpCon}) and the corresponding estimate from Theorem \ref{Thm:ModDevNonc} (ii)
hold true with $\xi_{q(n)}$ and $\gam_1'$ in place of $\xi_{n}$
and $\gam_1$, respectively. 
If we know that the limit $D^2$ exists (after this replacement) then all the other results
stated in Theorems \ref{Thm:ModDevNonc} and \ref{Thm2.4} also hold true with
$\frac1{\eta l^2}$ in place of $\frac1\eta$.

Consider, for instance, the case when $q_i$'s are polynomials and 
$q(n)=n^l$ for some $l\geq 2$,  namely, nonconventional sums of the form 
\begin{equation}\label{SepcNoncS}
\tilde S_N=\sum_{n=1}^NF(\xi_{p_1(n^l)},\xi_{p_2(n^l)},...,\xi_{p_\ell(n^l)})
\end{equation}
when all $p_i$'s are polynomials. Then the limit $D^2$ exists (see \cite{HK2}) and so all the results described above hold true.  
\end{remark}

%\subsection{stretched exponential mixing  rates}

%Prove in the case when we have stretched exponential $e^{-ax^\al}$ for some $a,\alpha>0$? 
%this is important since perhaps in several Young tower one may be able to approximate 
%the first hitting times tail probabilities by another process which only satisfies certain
%moderate deviations which will lead to stretched exponential mixing: The order of the 
%expressions changes (since, for instance, $(2k)!~(k!)^2e^k(\sqrt{2\pi k})^{-1}$).
%So I'll remark about it and then talk about stretched exponential
%rates in the "extensions"

%The same thing holds true in the setups from \ref{PrFun}.

%\subsection{processes in random environment}
%GENERAL IDEA: mention some applications to random dynamics??

\section{Additional results}\label{sec6}\setcounter{equation}{0}
\subsection{The CLT and Berry-Esseen type estimates}
We recall first the following result (see Corollary 2.1 in \cite{SaulStat}),
\begin{lemma}\label{BE LEM}
Let $W$ be a random variable. Suppose that there exist $\gam\geq0$
and $\Del>0$ so that for any $k\geq3$,
\[
|\Gam_k(W)|\leq (k!)^{1+\gam}\Del^{-(k-2)}.
\]
Let $\Phi$ be the standard normal distribution.
Then,
\[
\sup_{x\in\bbR}|P(W\leq x)-\Phi(x)|\leq c_\gam\Del^{-\frac1{1+2\gam}}
\]
where $c_\gam=\frac16\big(\frac{\sqrt 2}6\big)^{\frac1{1+2\gam}}$.
\end{lemma}
Note that when $|\Gam_k(W)|\leq C(k!)^{1+\gam}\Del^{-(k-2)},\,k\geq 3$ for some constant $C\geq1$ then the conditions
of Lemma \ref{BE LEM} are satisfied with $\Del C^{-1}$ in place of $\Del$.
This lemma together with the cumulants' estimates obtained in Theorem \ref{NoncCum} yields convergence
rates in the nonconventional CLT 
for $S_N/D\sqrt N$ which  (when $\eta=1$) are at best of order $N^{-\frac16}$, since in our circumstances $\Del$ is of
 order $N^{\frac12}$ and 
$\gam\geq1$, where in the case when $F$ is bounded we can take $\gam=1$.  
The rate $N^{-\frac16}$ is better than the ones obtained in \cite{HK1}, which is important since the rates obtained in \cite{Ha} and \cite{book} do not apply to the cases considered in Section \ref{sec5}.
Note that, in fact, 
we obtain here for the first time the CLT under condition (\ref{DecCorCond}) when $F$ has the form
(\ref{Product form}).

\begin{remark}
Consider the case discussed in Remark \ref{NonLineRem} when all $q_i$'s 
have the form $q_i(n)=p_i(n^l)$ for some polynomials $p_1,...,p_\ell$
and an integer $l\geq2$, namely the sums $N^{-\frac12}\tilde S_N$ where 
$\tilde S_N$ is defined in (\ref{SepcNoncS}). Then under Assumption \ref{ass1}
we obtain (when $D^2>0$) closer to optimal rates. Indeed, in these circumstances Theorem \ref{NoncCum} holds true with $\gam_1'=\frac1{\eta l^2}$ in place of 
$\gam_1=\frac1\eta$ and so, using the equality $\Gam_k(aW)=a^k\Gam_k(W),\,a\in\bbR$, 
we can apply now Lemma \ref{BE LEM} with $\gam=\gam_1'$ and $\Del$ of the form  $\Del=c\sqrt N$ and obtain rates 
of order $N^{-\frac1{2+4(\eta l^2)^{-1}}}$, which are better than $N^{-\frac16}$ when $\eta=1$.
\end{remark}

\subsection{Moment estimates of Gaussian type}
Theorem \ref{NoncCum} also implies the following
\begin{theorem}\label{MomThm}
Suppose that the conditions of Theorem \ref{NoncCum} hold true. 
Let $Z$ be a random variable which is distributed according to the standard normal law. 
Then for any $p\geq1$, 
\[
\big|\bbE(\bar S_N)^p-(\mathrm{Var}(S_N))^{\frac p2}\bbE Z^p\big|\leq 
(c_{0,1})^p(p!)^{1+\gam}\sum_{1\leq u\leq\frac {p-1}2}N^u\frac{p^u}{(u!)^2}
\]
where $c_{0,1}=\max(1,c_0)$,$\gam=\gam_1$ when Assumption \ref{ass1} holds true, $\gam=\gam_2$ when 
Assumption \ref{ass2} holds true and $c_0$, $\gam_1$ and $\gam_2$ are specified in Theorem \ref{NoncCum}.
\end{theorem}
\begin{proof}
The arguments below are based on the proof of Theorem 3 in \cite{Douk}.
By formula (1.53) in \cite{SaulStat}, for any $p\geq1$ and $N\in\bbN$,
\[
\bbE(\bar S_N)^p=\sum_{1\leq u\leq \frac p2}\frac1{u!}\sum_{k_1+k_2+...+k_u=p}\frac{p!}{k_1!\cdots k_u!}\Gam_{k_1}(\bar S_N)\cdots \Gam_{k_u}(\bar S_N).
\]
Let $1\leq u\leq\frac p2$.
When $k_i=1$ for some $1\leq i\leq u$ then $\Gam_{k_i}(\bar S_N)=\bbE\bar S_N=0$ and so 
the corresponding summand $\frac{p!\prod_{i=1}^u\Gam_{k_i}(\bar S_N)}{\prod_{i=1}^u(k_i!)}$ vanishes.
When $p$ is even and $u=\frac{p}{2}$ then the unique non-vanishing summand 
corresponds to the choice of $k_i=2,\,i=1,2,...,u$ and it equals $\big(\mathrm{Var}(S_N)\big)^{\frac p2}\bbE Z^p$. 
When $p$ is odd then $\bbE Z^p=0$, and therefore for any $p\geq 1$,
\begin{eqnarray*}
\big|\bbE(\bar S_N)^p-\big(\mathrm{Var}(S_N)\big)^{\frac p2}\bbE Z^p\big|\leq\\
\sum_{1\leq u\leq\frac {p-1}2}\frac1{u!}\sum_{k_1+k_2+...+k_u=p}\frac{p!}{k_1!\cdots k_u!}|\Gam_{k_1}(\bar S_N)\cdots \Gam_{k_u}(\bar S_N)|.
\end{eqnarray*}
Applying the H\"older inequality to Euler's $\Gamma$ function we obtain that 
$(k!)^p\leq (p!)^k$ for any integers  $k$ and $p$ so that $1\leq k\leq p$.   
Using Theorem \ref{NoncCum} we derive that 
\[
\frac{|\Gam_{k_1}(\bar S_N)\cdots \Gam_{k_u}(\bar S_N)|}{k_1!\cdots k_u!}\leq N^u 
c_0^{\sum_{i=1}^u k_i-2u}\big(\prod_{i=1}^u (k_i!)\big)^\gam\leq (c_{0,1})^p(p!)^{\gam}
\]
for any $1\leq k_1,...,k_u$ so that $\sum_{i=1}^uk_i=p$,
where $\gam$ is described in the statement of Theorem \ref{MomThm}. Thus, 
\[
\big|\bbE(\bar S_N)^p-\big(\mathrm{Var}(S_N)\big)^{\frac p2}\bbE Z^p\big|\leq
(c_{0,1})^p(p!)^{1+\gam}\sum_{1\leq u\leq\frac {p-1}2}\frac{\cN(u,p)}{u!}
\]
where 
\begin{eqnarray*}
\cN(u,p):=\big|\{2\leq k_1,...,k_u\leq p:\sum_{i=1}^uk_i=p\}\big|\leq\\
\big|\{1\leq k_1,...,k_u\leq p:\sum_{i=1}^uk_i=p \}\big|\leq \frac{p^u}{u!}.
\end{eqnarray*}
We conclude from the above estimates that for any integer $p\geq1$,
\[
\big|\bbE(\bar S_N)^p-\big(\mathrm{Var}(S_N)\big)^{\frac p2}\bbE Z^p\big|\leq 
(c_{0,1})^p(p!)^{1+\gam}\sum_{1\leq u\leq\frac {p-1}2}N^u\frac{p^u}{(u!)^2}
\]
and the proof of Theorem \ref{MomThm} is complete.
\end{proof}
We note that this theorem yields an appropriate Rosenthal type inequality for
the nonconventional sums $\bar S_N$ and that, in fact, makes the method of 
moments effective for them, which provides an additional proof of the nonconventional
central limit theorem. See Remarks 4 and 5 in \cite{Douk}, where
we also use (\ref{VarEst})(which implies that $N^{-1}\mathrm{Var}(S_N)$ converges to $D^2$ as $N\to\infty$). 
%the clt follows now by the method of moments...this is another proof!

%The following result will be used.%I think I still need this
%\begin{lemma}\label{lem3.17} 
%Let $X$ and $Y$ be two random variables defined on the same probability space. 
%Let $Z$ be a random variable with density $\rho$ bounded from above by some constant $c>0$. Then,
%\begin{eqnarray*}
%d_K(\cL(Y),\cL(Z))\leq 3d_K(\cL(X),\cL(Z))+
%4c\|X-Y\|_{L^\infty}\,\,\text{ and for any }\,\, b\geq 1, \\
%d_K(\cL(Y),\cL(Z))\leq 3d_K(\cL(X),\cL(Z))+(1+4c)\|X-Y\|_{L^b}^{1-\frac 1{b+1}}.
%\end{eqnarray*}
%\end{lemma}
%The second inequality is proved in Lemma 3.3 in \cite{HK2}, while 
%the proof of the first inequality
%goes in the same way as the proof of that
%Lemma 3.3, taking in (3.2) from there $\del=\|X-Y\|_{L^\infty}$. 


\begin{thebibliography}{Bow75}

\bibliographystyle{alpha}
\itemsep=\smallskipamount



%\bibitem{BoltExact}
%E. Bolthausen, {\em Exact convergence rates in some martingale central limit theorems}, Ann. Probab. 10 (1982), 672-688.

\bibitem{Jon}
J. Aaronson, M. Denker, {\em Local Limit Theorems for Gibbs-Markov Maps},
Stoch. Dyn. 1 (2001), 193-237.

 \bibitem{Bow}
R. Bowen, {\em Equilibrium states and the ergodic theory of Anosov diffeomorphisms},
 Lecture Notes in Mathematics, volume 470, Springer Verlag, 1975.
 
 
\bibitem{Br}
R.C. Bradley, {\em Introduction to Strong Mixing Conditions}, Volume 1, Kendrick Press, Heber City, 2007.


\bibitem{Chat1}
S. Chatterjee, {\em An introduction to large deviations for random graphs}, 
Bull. Amer. Math. Soc. 53 (2016) 617-642.


\bibitem{Chat2}
S. Chatterjee and A. Dembo, {\em Non linear large deviations}, Adv. Math. 299 (2016) 396-450.


\bibitem{Chaz}
J.R. Chazottes and S. Gou\"ezel, {\em Optimal concentration inequalities for dynamical systems}, 
Comm. Math. Phys. 316 (2012),  843-889.


\bibitem{Chaz2}
J.R. Chazottes, {\em Fluctuations of observables in dynamical systems: from limit theorems to concentration inequalities}, Nonlinear dynamics new directions, Vol. 11,  Nonlinear Syst. Complex., pages 47–85. Springer, 2015.



%\bibitem{CS}
%L.H.Y Chen and Q.M. Shao, {\em  Normal approximation under local dependence}, 
%Ann.Probab. 32 (2004), 1985-2028.

\bibitem{Ded}
J. Dedecker, P. Doukhan, G. Lang, J.R. León, S. Louhichi, C. Prieur, {\em Weak Dependence: With Examples and Applications}, Lecture Notes in Statistics, vol 190, Springer-Verlag, Berlin (2007).

\bibitem{DemZet}
A. Dembo and O. Zeitouni, {\em Large Deviations Techniques and Applications}, 2nd edn. Applications of
Mathematics, vol. 38. Springer, New York (1998).


\bibitem{Dor}
H. D\"oring and P. Eichelsbacher, {\em Moderate deviations via cumulants}, 
J. Theor. Probab.  26 (2013), 360-385.

\bibitem{Douk0}
P. Doukhan, {\em Mixing: Properties and Examples}, 
Lecture Notes in Statistics, Vol. 85, Springer, Berlin (1994).

\bibitem{Douk}
P. Doukhan and M.Neumann, {\em Probability and moment inequalities for sums of weakly dependent random variables, with applications}, Stochastic Process. Appl. 117 (2007), 878-903.

\bibitem{Fu}
H. Furstenberg, {\em Nonconventional ergodic averages},
Proc. Symp. Pure Math. 50 (1990), 43-56.

\bibitem{Gorc}
A.B. Grochakov, {\em Upper estimates for semi invariants of the sum of multi-indexed random variables}, 
Discrete Math. Appl. 5 (1995), 317-331.
 


%\bibitem{Gor}
%M.I. Gordin, {\em On the central limit theorem for stationary processes}, 
%Soviet Math. Dokl. 10 1174-1176.



%\bibitem{HH2}
%P. Hall, C.C. Heyde, {\em 
%Rates of convergence in the martingale central limit theorem}, 
% Ann. Probab. 9 (1981) 395-404.

\bibitem{HKllt}
Y. Hafouta and Yu. Kifer, {\em A nonconventional local limit theorem}, 
J. Theor. Probab. 29 (2016), 1524-1553.

\bibitem{HK1}
Y. Hafouta and Yu. Kifer, {\em Berry-Esseen type estimates
for nonconventional sums}, Stoch. Proc. Appl. 126 (2016), 2430-2464.

\bibitem{HK2}
Y. Hafouta and Yu. Kifer, {\em Nonconventional polynomial CLT}, 
Stochastics, 89 (2017), 550-591.

\bibitem{Ha}
Y. Hafouta, {\em Stein's method for nonconventional sums}, arXiv:1704.01094.

\bibitem{book}
Y. Hafouta and Yu. Kifer, {\em Nonconventional limit theorems and random dynamics}, 
World Scientific, Singapore, 2018.
  
\bibitem{HallHyde}
P.G. Hall and C.C. Hyde, {\em Martingale central limit theory and its application},
Academic Press, New York, 1980.

%\bibitem{Hyd}
% N.T.A. Haydn and Y. Psiloyenis, {\em Return times distribution for Markov towers with decay of correlations}, Nonlinearity 27 (2014), 1323-1349.
 
  
\bibitem{Hyd1}
N.T.A. Haydn and F. Yang, {\em Local escape rates for $\phi$-mixing dynamical systems},
arXiv preprint 1806.07148.

\bibitem{KalNe}
R.S. Kallabis and M.H. Neumann, {\em An exponential inequality under weak dependence},
Bernoulli 12 (2006) 333-350.


\bibitem{Ki2}
Yu. Kifer, {\em Nonconventional limit theorems },
Probab. Th. Rel. Fields, 148 (2010), 71-106.

\bibitem{Kif-LLN}
Yu. Kifer, {\em A nonconventional strong law of large numbers and fractal dimensions of some multiple
recurrence sets}, Stoch. Dynam. 12 (2012), 1150023.


\bibitem{Ki3}
Yu. Kifer, {\em Strong approximations for nonconventional sums and almost sure limit theorems}, 
Stochastic Process. Appl. 123 (2013), 2286-2302.


\bibitem{KV}
Yu. Kifer and S.R.S Varadhan, {\em Nonconventional
limit theorems in discrete and continuous time via martingales},  Ann.
Probab. 42 (2014), 649-688.

\bibitem{KV1}
Yu. Kifer and S.R.S. Varadhan, {\em Nonconventional large deviations theorems},
Probab. Th. Rel. Fields 158 (2014), 197-224.



\bibitem{Mil}
V.D. Milman and G. Schechtman, {\em Asymptotic theory of finite-dimensional normed spaces}, Lecture
Notes in Mathematics, Vol. 1200, Berlin: Springer-Verlag, 1986, (With an appendix by M.
Gromov).

\bibitem{MSU}
V. Mayer, B. Skorulski and M. Urba\'nski, {\em Distance expanding random mappings,
thermodynamical formalism, Gibbs measures and fractal geometry},
Lecture Notes in Mathematics, vol. 2036 (2011), Springer.
 
\bibitem{SaulStat}
L. Saulis  and V.A. Statulevicius, {\em Limit Theorems for Large Deviations}, 
Kluwer Academic, Dordrecht, Boston, 1991.





\bibitem{Young1}
L.S. Young, {\em Statistical properties of dynamical systems with some hyperbolicity}, 
 Ann. Math. 7 (1998) 585-650.

\bibitem{Young2}
L.S. Young, {\em  Recurrence time and rate of mixing}, 
Israel J. Math. 110 (1999) 153-88.







\end{thebibliography}
\end{document}